\pdfoutput=1
\documentclass[a4paper]{article}
\usepackage{amsfonts}
\usepackage{mathtools}
\usepackage{amsthm, amssymb, amsfonts, enumerate}
\usepackage{tikz-cd}
\usepackage{spectralsequences}
\usepackage{geometry}
\usetikzlibrary{matrix,positioning,arrows.meta}
\usetikzlibrary{arrows}

\newcommand{\rrightarrow}{\mathrel{\mathrlap{\rightarrow}\mkern1mu\rightarrow}}

\DeclareMathOperator{\Diff}{Diff}
\DeclareMathOperator{\Emb}{Emb}
\DeclareMathOperator{\Isom}{Isom}

\DeclareMathOperator{\Fr}{Fr}
\DeclareMathOperator{\GL}{GL}
\DeclareMathOperator{\SO}{SO}
\newcommand{\interior}[1]{\smash{\mathring{#1}}}
\DeclareMathOperator{\Norm}{Norm}

\DeclareMathOperator{\Dih}{Dih}
\DeclareMathOperator{\Stab}{Stab}
\DeclareMathOperator{\image}{im}

\newcommand{\hq}{/\!\!/} 

\newcommand{\Or}{\operatorname{O}}

\newtheorem{theorem}{Theorem}[section]

\newtheorem{proposition}[theorem]{Proposition}
\newtheorem{lemma}[theorem]{Lemma}
\newtheorem{corollary}[theorem]{Corollary}
\theoremstyle{definition}
\newtheorem{definition}[theorem]{Definition}
\newtheorem{notation}[theorem]{Notation}

\newtheorem{example}[theorem]{Example}
\newtheorem{remark}[theorem]{Remark}

\SseqNewClassPattern{myclasspattern}{
(0,0);
(-0.3,0)(0.3,0);
(-0.4,0.3)(-0.3,-0.3)(0.4,0.3);
}

\newcommand{\fakeenv}{} 
\newenvironment{restate}[2]  
{
  \renewcommand{\fakeenv}{#2}   
  \theoremstyle{plain}
  \newtheorem*{\fakeenv}{#1~\ref{#2}} 
  \begin{\fakeenv}  
}
{ \end{\fakeenv} }

\usepackage{hyperref}

\begin{document}
\title{Cohomology of the diffeomorphism group of the connected sum of two generic lens spaces}
\author{Zoltán Lelkes}
\date{}

\maketitle

\begin{abstract}
    We consider the connected sum of two three-dimensional lens spaces $L_1\#L_2$, where $L_1$ and $L_2$ are non-diffeomorphic and are of a certain "generic" type.
    Our main result is the calculation of the cohomology ring $H^\ast(B\Diff(L_1\#L_2);\mathbb{Q})$, where $\Diff(L_1\#L_2)$ is the diffeomorphism group of $M$ equipped with the $C^\infty$-topology.
    We know the homotopy type of the diffeomorphism groups of generic lens spaces this, combined with a theorem of Hatcher forms the basis of our argument.
\end{abstract}
\section{Introduction}
For a smooth 3-manifold $M$, let $\Diff(M)$ be its diffeomorphism group endowed with the $C^\infty$-topology.
The space $B\Diff(M)$ classifies smooth $M$-bundles, in the sense that concordance classes of smooth $M$-bundles over a space $X$ are in bijection with homotopy classes of maps $X\to B\Diff(M)$, where this bijection is given by pulling back the universal smooth $M$-bundle over $B\Diff(M)$, see \cite{galat19}.
Therefore, the cohomology of $B\Diff(M)$ gives characteristic classes of smooth $M$-bundles.

The 3-dimensional lens space $L(m, q)$ is the quotient of $S^3\subseteq \mathbb{C}^2$ by the action of $C_m$, the cyclic group of order m, induced by multiplication with $\xi_m$ in the first coordinate and with $\xi_m^q$ in the second coordinate, where $\xi_m$ is the mth root of unity.
These inherit the structure of a (Riemannian) 3-manifold and in fact they are prime 3-manifolds.
We call a 3-dimensional lens space a generic lens space if $m>2$, $1<q<\frac{m}{2}$, and $q^2\not\equiv \pm 1 \mod m$.
Generic lens spaces do not admit any orientation reversing diffeomorphisms, see \cite{mccul00}.
In this text, we will always take cohomology with rational coefficients and in order to make notation more convenient we omit them.
We prove the following main result.

\begin{restate}{Theorem}{main result}
    Let $L_1$ and $L_2$ be two non-diffeomorphic  two generic lens spaces.
    \[H^\ast(B\Diff(L_1\#L_2))\cong \mathbb{Q}[\mu^2, \eta^2, \nu^2, \vartheta^2] / (\mu^2\eta^2, \nu^2\vartheta^2, \mu^2+\eta^2-\nu^2-\vartheta^2).\]
\end{restate}
We compute the mapping class group of $L_1\#L_2$ as well, this computation plays a crucial role in showing the main result.
\begin{restate}{Theorem}{thm: mapping class group}
    Let $L_1$ and $L_2$ be two non-diffeomorphic generic lens spaces.
    \[\pi_0 (\Diff(L_1\#L_2)) \cong C_2\times C_2.\]
\end{restate}

To expand on Theorem \ref{main result} let us give a rundown of where the generators $\mu$, $\eta$, $\nu$, $\vartheta$ in  ultimately arise from.
By \cite{Hong11} for a generic lens space $L$, the inclusion 
$\Isom(L)\hookrightarrow \Diff(L)$ is a weak equivalence, where $\Isom(L)$ is the isometry group of $L$.
The isometry group of a generic lens space is calculated in \cite{mccul00}. 
It is shown there that $\Isom(L)_0$ is covered m-fold by an $\SO(2)\times \SO(2)$ subgroup of $\SO(4)$, where $G_0\triangleleft G$ denotes the path component of the identity in the topological group $G$.
Let us denote by $\mathbb{Q}[e\otimes 1, 1\otimes e]$ the cohomology ring of $\SO(2)\times \SO(2)$ where the two generators are the Euler classes pulled back along the projections.
In the cohomology ring of $B\Diff(L_1)_0$, we denote $\mu$ the preimage of $e\otimes 1$ and $\eta$ the preimage of $1\otimes e$.
Similarly for $B\Diff(L_2)_0$, $\nu$ denotes the preimage of $e\otimes 1$ and $\vartheta$ denotes the preimage of $1\otimes e$.
The theorem of Hatcher referenced in the abstract is remarked in \cite{Hatch81} and states that in case $M$ is the connected sum of two prime 3-manifolds, then $\Diff(M)$ deformation retracts onto $\Diff(M, S^2)$ where $S^2\subseteq M$ is a copy of the non-trivial 2-sphere in $M$.
We calculate $H^\ast(B\Diff(L_1\#L_2, S^2)_0)$ via considering the restrictions to $B\Diff(L_1\setminus \interior{D^3})_0$ and $B\Diff(L_2\setminus \interior{D^3})_0$.
We show that $B\Diff_\text{pt}(L)_0 \simeq B\Diff(L\setminus\interior{D^3})_0$, where $\Diff_\text{pt}(L)_0$ is the subgroup of $\Diff(L)_0$ consisting of those diffeomorphisms that leave a given point $\text{pt}\in L_1\#L_2$ fixed.
In the cohomology of $B\Diff_\text{pt}(L)_0$ we pull back the generators from the generators of $B\Diff(L)_0$ via the inclusion.
Finally, note that $H^\ast(B\Diff(L_1\#L_2))$ is the subring $H^\ast(B\Diff(L_1\#L_2)_0)^{\pi_0\Diff(L_1\#L_2)}$.
For more details on this and for an overview of the proof, see Section \ref{strategy section}.

\subsection*{Comparison with previous work}
In dimension two, the Madsen-Weiss theorem \cite{MadsenWeiss07} proves the Mumford conjecture and describes the cohomology of $B\Diff(F)$ in a stable range for $F$, a smooth, compact, connected and oriented surface.

In high dimensions, Randal-Williams and Galatius \cite{OscarSoren17} show an analogue of the Madsen–Weiss theorem for any simply-connected manifold of dimension $2n\geq 6$.

In dimension 3 most of the work focuses on prime manifolds. Hatcher proved the Smale conjecture $\Diff(S^3)\simeq O(4)$ in \cite{Hatch83} and $\Diff(S^1\times S^2)\simeq O(2)\times O(3)\times \Omega O(3)$ in \cite{Hatch81}.
For Haken 3-manifods, by the work of Waldhausen \cite{Waldh68}, Hatcher \cite{Hatch76}, and Ivanov \cite{Ivanov79} the calculations of the homotopy types of $\Diff(M)$ largely reduce to those of the mapping class group.
A notable exception is \cite{bamler19} where they show the generalized Smale conjecture for all 3-dimensional spherical spaces, as well as $\Diff(\mathbb{R}P^3\#\mathbb{R}P^3)\simeq \Or(1)\times \Or(2)$.

In \cite{jan24} Boyd, Bregman, and Steinebrunner show that for a compact, orientable 3-manifold $M$, $B\Diff(M)$ is of finite type.
Their paper is where the outline of the arguments in this work originates.
In an upcoming paper they aim to calculate the rational cohomology ring
 of $B\Diff((S^1 \times S^2)^{\#2})$. 

In most cases when we know the homotopy type of $\Diff(M)$, if $\pi_0\Diff(M)$ is finite, it turns out to be that of a compact Lie group.
However, this is not the case for $L_1\#L_2$ where $L_1$ and $L_2$ are non-diffeomorphic generic lens spaces.
\begin{corollary}
    Let $L_1$ and $L_2$ be non-diffeomorphic generic lens spaces. $B\Diff(L_1\#L_2)$ is not weakly equivalent to the classifying space of a compact Lie group.
\end{corollary}
This is a consequence of Theorem \ref{main result} and Hopf's theorem (see e.g. \cite[Theorem 1.81]{Felix08}).
The latter states that for any $G$ compact Lie group, $H^\ast(BG_0)$ is a free polynomial ring on even generators.
Furthermore, $H^\ast(BG) \cong H^\ast(BG_0)^{G/G_0}$ (see e.g. \cite[Proposition 3G.1]{Hatch22}).
This means in particular that $H^\ast(BG)$ is an ideal domain, while $H^\ast(B\Diff(L_1\#L_2))$ is not by Theorem \ref{main result}.

\subsection*{Acknowledgements}
This project has grown out of my master's thesis, which I wrote under the supervision of Jan Steinebrunner.
I cannot thank him enough for his insights and ideas.
Writing both the thesis and this paper at every turn he has been there to provide guidance; it has truly been a great experience working with him.

\section{Background}\label{the setting}

\subsection{Lens spaces and their isometries}

We concern ourselves with 3-dimensional lens spaces, these are manifolds $L(m, q)$ for coprime $m, q\in \mathbb{N}$ such that $L(m,  q)$ is the quotient of $S^3\subseteq \mathbb{C}$ by the action generated by multiplication in the first coordinate by $e^\frac{2\pi i}{m}$ and  in the second by $e^\frac{2\pi i q}{m}$.
Two lens spaces $L(m_1, q_1)$ and $L(m_2, q_2)$ are diffeomorphic if and only if $m_1 = m_2$ and $q_1+q_2 \equiv 0 \mod m_1$ or $q_1q_2\equiv 1 \mod m_1$.
This is shown for example in \cite[Theorem 2.5]{Hatch23}.

An irreducible 3-manifold is a 3-dimensional manifold in which every embedded 2-sphere bounds a 3-disc.
A consequence of the Poincaré conjecture is that a connected, compact, orientable 3-manifold $M$ is irreducible if and only if $\pi_2(M)$ is trivial.
Since any 3-dimensional lens space is covered by the 3-sphere its second homotopy group is zero and thus all 3-dimensional lens spaces are irreducible.

By explicitly considering the cellular structure of $L(m, q)$ its rational cohomology can be shown to be $\mathbb{Q}$ in degrees $0$ and $3$ and trivial in all other degrees.
The quotient map $S^3\to L(m, q)$ induces an isomorphism on rational cohomology, since it is injective in top degree as it is a covering.

We take the unique metric on $L(m, q)$ that makes the covering $S^3 \to L(m, q)$ a Riemannian covering when considering the standard metric on $S^3$, such a metric exists as the action of $C_m$, a discrete subgroup of the isometry group of $S^3$, is free.

Recall the Smale conjecture proven by Hatcher in \cite{Hatch83}.
\begin{theorem}\label{thm: Smale conjecture}
    The inclusion $\Or(4)\cong\Isom(S^3)\hookrightarrow\Diff(S^3)$ is a weak equivalence, where $\Isom(S^3)$ denotes the group of isometries of $S^3$ when endowed with the standard Riemannian metric.
\end{theorem}

The diffeomorphism groups of these lens spaces are also well understood, since the generalized Smale conjecture holds for this class of 3-manifolds.
This is shown by Hong, Kalliongis, McCullough, and Rubinstein in \cite{Hong11}.

\begin{theorem}\label{thm: generalized smale conj}
    For any 3-dimensional lens space $L(m, q)$ with $m>2$, the inclusion of the isometry group into the diffeomorphism group of $L(m, q)$, $\Isom(L(m, q)) \hookrightarrow \Diff(L(m, q))$ is a homotopy equivalence.
\end{theorem}

McCullough in \cite{mccul00} presents a calculation of $\Isom(L(m, q))$.
He uses the unit quaternion group structure on $S^3$, letting $S^3=\{z_0 + z_1j | z_0,\,z_1\in\mathbb{C}\,s.t.\,|z_0|^2 + |z_1|^2 = 1 \}$ with the convention $zj = j\overline{z}$.  
The isometries are described using the following double covering by $S^3\times S^3$ of $\SO(4)$
\[\begin{tikzcd}[row sep=tiny]
	{F\colon S^3\times S^3} & {\SO(4)} \\
	{(q_1, q_2)} & {(q\mapsto q_1 q q_2^{-1}).}
	\arrow[from=1-1, to=1-2]
	\arrow[maps to, from=2-1, to=2-2]
\end{tikzcd}\]

\begin{enumerate}
    \item Denote $S^1 = \{z_0 \in \mathbb{C}\,|\, |z_0| = 1\} < S^3$ (i.e. the elements with no $j$ term), $\xi_k = e^\frac{2\pi i}{k} \in S^1$, and $C_k = \langle\xi_k\rangle$.
    \item Denote $\Dih(S^1\tilde{\times}S^1) = \langle F(S^1\times S^1), F(j, j)\rangle$ the subgroup of $\SO(4)$. 
    It may be described as the semidirect product $(S^1\tilde{\times}S^1)\rtimes C_2$, where $C_2$ acts by conjugation on each coordinate and $S^1\times S^1 = (S^1\times S^1)/\langle (-1, -1)\rangle$.
\end{enumerate}
The key to his approach lies in the following lemma, the proof of which we leave to the reader.
\begin{lemma}\label{lem: the descenting isometries}
    Let $G<\SO(4)$ be a finite subgroup acting on $S^3$ freely, such that its action is induced by the action of $\SO(4)$. If $M = S^3/G$, then $\Isom^{+}(M) \cong \Norm(G)/G$ where $\Norm(G)$ is the normalizer of $G$ in $\SO(4)$ and $\Isom^{+}(M)$ is the group of orientation preserving isometries of $M$.
\end{lemma}
In our case the $C_m$ action which we quotient $S^3$ by to gain $L(m, q)$ is described as the subgroup of $\SO(4)$ generated by $F(\xi_{2m}^{q+1}, \xi_{2m}^{q-1})$.
\begin{definition}
    A \textit{generic lens space} is a 3-dimensional lens space $L(m, q)$ such that $m>2$, $1<q<\frac{m}{2}$, and $q^2\not\equiv \pm 1 \mod m$.
\end{definition}
It is an important fact for us that generic lens spaces do not admit orientation reversing homeomorphisms, this comes from \cite[Proposition 1.1]{mccul00}.

Based on $m$ and $q$ the isometry group $\Isom(L(m, q))$ may be one of $8$ group and all generic lens spaces have isometry groups isomorphic to $\Dih(S^1\tilde{\times}S^1)/\langle F(\xi_{2m}^{q+1}, \xi_{2m}^{q-1})\rangle$.
Generic lens spaces are generic in the sense that given $m$, the ratio of possible choices of $1\leq q\leq m$ yielding  \[\Isom(L(m, q)) \cong \Dih(S^1\tilde{\times}S^1)/\langle F(\xi_{2m}^{q+1}, \xi_{2m}^{q-1})\rangle\] to $m$ tends to $1$ as $m$ tends to infinity.
\subsection{Fiber sequences of diffeomorphism groups}
Let us fix some notation for different subgroups of the diffeomorphism group of a manifold.
We always allow manifolds to have boundary.
\begin{definition}\label{def: diffeo groups notation}
    Let $M$ be a 3-manifolds, $V$ a manifold, and $U\subseteq M$ a submanifold.
    \begin{enumerate}
        \item $\Emb(V, M)\subseteq C^\infty(V, M)$ is the subset consisting of the embeddings of $V$ into $M$.
        \item $\Diff_\partial (M) = \{\varphi \in \Diff(M) \,|\, \forall x \in \partial M,\, \varphi(x) = x\}$.
        \item $\Diff_U(M) = \{\varphi \in \Diff(M) \,|\, \forall x \in U,\, \varphi(x) = x\}$.
        \item $\Diff(M, U) = \{\varphi \in \Diff(M) \,|\, \varphi(U) = U\}$.
        \item We often assume a Riemannian metric on $M$ and denote the group of isometries of $M$ by $\Isom(M)$.
    \end{enumerate}
    For all the groups $G$ above, we use the notation $G^+$ to denote the subset consisting of only orientation preserving maps, in case $M$ and $V$ are orientable, and if $V$ is codimension one we use the notation $\Emb^+(V, M)$ for orientation preserving embeddings.
    Furthermore, for all topological groups $G$ we will denote by $G_0$ the path component of the identity in $G$.
\end{definition}
To derive our fiber sequences we will rely on the notion of local retractileness defined as in \cite{Canter17}.
\begin{definition}
    Let $G$ be a topological group. 
    A \textit{$G$-locally retractile} space $X$ is a topological space with a continuous $G$-action, such that for all $x\in X$ there exists an open neighborhood $U\subseteq X$ of $x$ and a map $\xi\colon U \to G$, such that for all $y\in U$, $y = \xi(y).x$. 
    In this situation $\xi$ is a \textit{$G$-local retraction around $x$}.
\end{definition}
In this case locally $X$ is a retract of $G$, but a $G$-local retraction around $x$ is in fact a local section of the map $G\to X$ sending $g$ to $g.x$.
\begin{example}\label{eg: S^3 is SO(4) locally retractile}
    $S^3$ is an $\SO(4)$-locally retractile space.
    Given some base-point $q_0\in S^3$ we can write down an $\SO(4)$-local retraction around $q_0$ via $\xi\colon S^3\to \SO(4)$ with $\xi(q) = F(q, q_0)$.
\end{example}
From now on, we will always assume that actions of topological groups are continuous.
The following is a combination of lemmas from \cite[Lemma 2.4, 2.5, 2.6]{Canter17} except for point (4) which follows by choosing some path between points and then covering it by a finite number of opens and applying local retractileness.
\begin{lemma} \label{local retractileness}
    Let $G$ be a topological group and $E$ and $X$ spaces with a $G$-action, and let $f\colon E \to X$ be a $G$-equivariant map. 
    \begin{enumerate}[(1)]
        \item  If $X$ is $G$-locally retractile, then $f$ is a locally trivial fibration.
        \item If $f$ has local sections and $E$ is $G$-locally retractile, then $X$ is also $G$-locally retractile.
        \item Let $X$ be locally path connected and $G$-locally retractile. If $H<G$ is a subgroup containing the path component of the identity, then $X$ is also $H$-locally retractile.
        \item If $X$ is path connected and $G$-locally retractile, then the action of $G$ is transitive.
    \end{enumerate}
\end{lemma}
The following theorem proved by Lima in \cite{Lim64}, originally due to Palais and Cerf, implies that $\Emb(V, M)$ is $\Diff(M)$-locally retractile in case $V$ is compact, where the action on $\Emb(V, \interior{M})$ is given by post-composition.
\begin{theorem}\label{Emb is locally retractile}
    Let $M$ be a $C^\infty$-manifold, and $V\subseteq \interior{M}$ a compact submanifold. The space $\Emb(V, \interior{M})$ is $\Diff(M)$-locally retractile.
\end{theorem}
This provides us with the Palais fiber sequence.
Let $M$ be a $C^\infty$-manifold, $V\subseteq \interior{M}$ a compact submanifold. 
There is a fiber sequence of the form
\begin{equation}\label{eq: Palais fib seq}
    \Diff_V(M) \hookrightarrow \Diff(M) \to \Emb(V, \interior{M}).
\end{equation}
Pulling back the Palais fiber sequence gives the following lemma:
\begin{lemma}\label{submnfld fib seq}
    Given a compact submanifold $V\subseteq \interior{M}$ there is a fiber sequence 
    \[\Diff_V(M)\to \Diff(M, V) \to \Diff(V).\]
    Furthermore, for $\Diff^\prime(V)$ the space of those diffeomorphisms of $V$ that can be extended to a diffeomorphism of $M$ we have that the map $\Diff(M, V)\to \Diff^\prime(V)$ is a $\Diff_V(M)$-principal bundle.
\end{lemma}
The last point about the map $\Diff(M, V)\to \Diff^\prime(V)$ being a $\Diff_V(M)$-principal bundle is especially useful when considering in tandem with the following lemma from \cite[Corollary 2.11 (2)]{bonat20}.
\begin{lemma}\label{ses delooped}
    For $i = 1, 2, 3$ let $G_i$ be a topological group and and $S_i$ a space with a $G_i$-action. Let $1\to G_1\to G_2 \overset{\phi}{\to}G_3\to 1$ be a short exact sequence of groups such that $\phi$ is a $G_1$-principal bundle.
    If $S_1\to S_2\to S_3$ is a fiber sequence of equivariant maps, then the induced maps on quotients form a homotopy fiber sequence
        \[S_1\hq G_1 \to S_2\hq G_2 \to S_3\hq G_3.\]
\end{lemma}
We will use two special cases of this lemma, both of them are well-known results, one is the case where $S_1=S_2=S_3=\text{pt}$, which allows us to deloop the short exact sequence of groups into a homotopy fiber sequence $BG_1\to BG_2\to BG_3$, the second is where $S_1 = S_2 = X$, $S_3= \text{pt}$ and $G_1 = 1$, $G_2=G_3 = G$, which gives for all $G$-spaces $X$ a homotopy fiber sequence $X\to X\hq G \to BG$.
\begin{remark}
    Let $1\to G_1\to G_2 \overset{p}{\to}G_3\to 1$ be a short exact sequence of topological groups. 
    $G_3$ is a $G_2$-locally retractile space with respect to the induced action from $p$, if and only if $p$ is a $G_1$-principal bundle.
    In this case we call the short exact sequence a principal short exact sequence.
\end{remark}

Cerf in \cite{Cerf61} showed the contractibility of collars, the following formulation of it comes from \cite[Theorem 2.6]{jan24}.
\begin{theorem}\label{contractable collars}
    The space of collars
    \[\Emb_{\partial M}(\partial M \times I, M) = \{\iota \in \Emb(\partial M \times I, M) \,|\, \left.\iota\right|_{\partial M} = \text{id}_{\partial M}\}\]
    is weakly contractible, where $\partial M \times I$ is a tubular neighborhood of $\partial M$.
    
    As a consequence we have that the subgroup inclusion \[\Diff_U(M)\hookrightarrow\Diff_{\partial U}(M\setminus \interior{U})\]
    is a weak equivalence for a codimension 0 submanifold $U\subseteq \interior{M}$.
\end{theorem}
The next lemma, a consequence of the \textit{homotopical orbit stabilizer lemma}, \cite[Lemma 2.10]{jan24} .
\begin{lemma}\label{lem: id path component homotopical orbit stabilizer}
    Let $X$ be a path connected $G$-locally retractile space such that the $G$ action on $X$ is transitive, and let $x\in X$.
    Consider the inclusion $\{x\}\hookrightarrow X$, this is equivariant with respect to $\Stab_G(x)_0\hookrightarrow G_0$,
    where $G_0 \triangleleft G$ is the path component of the identity in $G$ and $\Stab_G(x) < G$ is the stabilizer group of $x$ in $G$.
    
    If the inclusion of $\Stab_G(x)$ into $G$ induces a bijection on path components, then the equivariant inclusion of $x$ into $X$ induces a weak equivalence, in fact a homeomorphism for the right models of the classifying spaces,
    \[B\Stab_G(x)_0 \overset{\simeq}{\to}X\hq G_0.\]
    Moreover, there is a homotopy fiber sequence
    \[X\to B \Stab_G(x)_0 \to BG_0.\]
\end{lemma}
\begin{proof}
    By Lemma \cite[Lemma 2.10]{jan24}, the map
    \[\begin{tikzcd}[cramped, row sep=small]
    	{\Stab_G(x)} & G \\
    	\{x\} \arrow[loop above, out=120, in=70, distance=15] & X \arrow[loop above, out=120, in=70, distance=15]
    	\arrow[hook, from=1-1, to=1-2]
    	\arrow[hook, from=2-1, to=2-2]
    \end{tikzcd}\]
    induces a weak equivalence $B\Stab_G(x) \overset{\simeq}{\to}X\hq G$, which is in fact a homeomorphism for the right models of the classifying spaces
    We have to see that 
    \[\Stab_{G}(\iota)_0\hookrightarrow\Stab_{G_0}(\iota) = G_0\cap\Stab_{G}(x)\]
    is a surjection.
    The assumption that $\Stab_G(x)\hookrightarrow G$ induces a bijection on path components means that any $g\in \Stab_{G}(x)$ is in $\Stab_{G}(x)_0$ if and only if it is connected to the identity in $G$, i.e. is in $G_0$.
\end{proof}

\begin{theorem} \label{embeddings of discs are framings}
    If $M$ is an $m$-dimensional manifold, then the differential at $0$ gives a weak equivalence $\Emb(D^m, M)\overset{\simeq}{\to}\Fr(TM)$.
\end{theorem}
\begin{lemma}\label{lem: cut out disc}
    Let $M$ be a closed 3-manifold and $D\subseteq M$ an embedded 3-disc.
    Denote 
    \[\Diff^{\Or}(M, D) = \{\varphi\in \Diff(L, D)\,|\, \left.\varphi\right|_{D}\in \Or(3)\subseteq \Diff(D)\}.\]
    The maps 
    \[\Diff(M\setminus \interior{D})\leftarrow \Diff^{\Or}(M, D) \to \Diff_{x}(M)\]
    are weak equivalences, where $x\in D$ is its center point.
\end{lemma}
\begin{proof}
    The map $\Diff^{\Or}(M, D)\to \Diff(M\setminus \interior{D})$ is the pullback of the map $\Or(3)\to \Diff(\partial(M\setminus \interior{D}))$ along the restriction $\Diff(M\setminus \interior{D})\to \Diff(\partial(M\setminus \interior{D}))$.
    By the Smale theorem, the map $\Or(3) \to \Diff(S^2)\cong \Diff(\partial(M\setminus \interior{D}))$ is a weak equivalence.

    The map $\Diff^{\Or}(M, D)\to \Diff_{x}(M)$ is a weak equivalence as it is a pullback of the map $\Or(3)\to\Emb_{\{x\}}(D^3, M)$ that is given by acting through precomposition by an element of $\Or(3)$ viewed as a diffeomorphism of $D^3$ on the embedding of $D$.
    Here $\Emb_{\{x\}}(D^3, M) = \{i \in \Emb(D^3, M)\, |\, i(0) = x\}$.
    Taking the derivative at $x$ gives a weak equivalence $\Emb_{\{x\}}(D^3, M)\to \GL_3(\mathbb{R})$ and this means that as $\GL_3(\mathbb{R})$ retracts onto $\Or(3)$, the composition with $\Or(3)\to\Emb_{\{x\}}(D^3, M) $ is a weak equivalence and we conclude using the 2 out of 3 property.
\end{proof}

\section{Setup}
\subsection{The main homotopy fiber sequence}
There is a theorem of Hatcher, remarked in \cite{Hatch81}, also proven in \cite[Theorem 3.21]{jan24} stating:
\begin{theorem}\label{theorem of Hatcher}
    Let $M$ be a connected sum of two irreducible manifolds that are not diffeomorphic to $S^3$. 
    If $S\subseteq M$ is the 2-sphere these irreducible pieces are joined along, then the inclusion 
    $\Diff(M, S) \hookrightarrow \Diff(M)$ is an equivalence.
\end{theorem}
From now on we set $M\cong L_1\#L_2$ for two generic lens spaces, so that $L_1\not \cong L_2$.
Fix a 2-sphere $S$ in $M\cong L_1\#L_2$ is such that $M\setminus N(S) \cong L_1\setminus\interior{D^3} \sqcup L_2\setminus\interior{D^3}$ where $N(S)$ is an open tubular neighborhood of $S$.
As $L_1\not\cong L_2$, $\Diff(M)\simeq \Diff(M, S)\cong \Diff(M, L_2\setminus\interior{D^3})$.
Consider the following exact sequence of topological groups,
\begin{equation}\label{main fib seq w.o. delooping}
    \Diff_{L_2\setminus\interior{D^3}}(M)\to \Diff(M, L_2\setminus\interior{D^3}) \overset{p}{\to} \Diff(L_2\setminus\interior{D^3}).
\end{equation}
By Lemma \ref{submnfld fib seq}, to see that this is a principal short exact sequence, we need the second map to be surjective.
However as a consequence of contractability of collars, we have the following lemma:
\begin{lemma}\label{lem: extendability based on boundary}
    Let $V\subseteq M$ be a codimension zero submanifold of M and $\varphi\in\Diff(V)$.
    There is some $f\in \Diff(M, V)$ such that  $\left.f\right|_V = \varphi$ if and only if there is some $\psi\in \Diff(M, V)$ such that \[[\left.\psi\right|_{\partial V}] = [\left.\varphi\right|_{\partial V}]\in\pi_0\Diff(\partial V).\]
    This says that the extendability of $\varphi$ only depends on $[\left.\varphi\right|_{\partial V}]\in \pi_0\Diff(\partial V)$.
\end{lemma}
On one hand $\pi_0 \Diff(\partial L_2\setminus\interior{D^3}) \cong \pi_0 \Diff(S^2) \cong \pi_0 \Or (3)\cong C_2$, where under the last isomorphism orientation preserving diffeomorphisms are mapped to $+1$ and orientation reversing diffeomorphisms are mapped to $-1$.
On the other hand, generic lens spaces do not admit orientation reversing homeomorphisms, \cite[Proposition 1.1]{mccul00}, and therefore for all $\varphi \in \Diff(\partial L_2\setminus\interior{D^3})$, $[\left.\varphi\right|_{\partial L_2\setminus\interior{D^3}}] = [\text{id}]\in \pi_0 \Diff(\partial L_2\setminus\interior{D^3})$.
This means Lemma \ref{lem: extendability based on boundary} implies that the short exact sequence (\ref{main fib seq w.o. delooping}) is a principal short exact sequence.
This in particular means that by Lemma \ref{ses delooped} we can deloop this to a homotopy fiber sequence as follows:
\begin{equation}\label{main fib seq}
    B\Diff_{L_2\setminus\interior{D^3}}(M)\to B\Diff(M, L_2\setminus\interior{D^3}) \to B\Diff(L_2\setminus\interior{D^3}).
\end{equation}
Let us inspect the outer terms of (\ref{main fib seq}).
Contractability of collars implies that $\Diff_{L_2\setminus\interior{D^3}}(M)\simeq \Diff_\partial(L_1\setminus\interior{D^3})$.
Applying it again yields $\Diff_\partial(L_1\setminus\interior{D^3})\simeq \Diff_{D^3}(L_1)$.
Furthermore applying Lemma \ref{lem: cut out disc} we get $\Diff(L_2\setminus\interior{D^3}) \simeq \Diff_{\text{pt}}(L_2)$.
This means that to get the terms in the Leray-Serre spectral sequence induced by (\ref{main fib seq}), we just have to calculate the cohomology of $B\Diff_{D^3}(L_1)$ and $B \Diff_{\text{pt}}(L_2)$.

\subsection{Strategy}\label{strategy section}
Let us go over our strategy for the proof before we get to the details.
By Theorem \ref{theorem of Hatcher} $\Diff(M, S)\simeq \Diff(M)$ and we want to compute the cohomology of the classifying space of $G = \Diff(M, S)$.
Our strategy to calculate the cohomolgy of $BG$ is using the homotopy fiber sequence
\[BG_0\to BG \to B\pi_0G\]
where $G_0$ is the path component of the unit in $G$.
Since the $E_2$-page is twisted, one has to determine the action of $\pi_1 BG\cong \pi_0 G$ on the cohomolgy of $BG_0$ in order to figure out the cohomology of $BG$.
If we can do this, and assuming that $G_0$ is a finite group, we obtain that
\[H^\ast(BG) \cong H^\ast(BG_0)^{\pi_0 G}.\]
This means we need to calculate $\pi_0 \Diff(M, S)$, $H^\ast(B\Diff(M, S)_0)$, and the action.

We calculate the cohomology groups $H^k(B\Diff(M, S)_0)$ using the cohomological Leray-Serre spectral sequence associated to the homotopy fibers sequence (\ref{main fib seq}), this will turn out to collapse on the second page.
However this does not tell us the ring structure.
In order to calculate that we use the map induced by the product of the restrictions
\[H^\ast(B\Diff(L_2\setminus\interior{D^3})_0 \times B\Diff(L_1\setminus\interior{D^3})_0)\to H^\ast(B\Diff(M, S)_0).\]
We show that the kernel of this map contains a specific ideal, and then as we know the dimensions of $H^k(B\Diff(M, S)_0)$ as a $\mathbb{Q}$-vector space for each $k$, we can conclude that the kernel is in fact equal to that ideal.

In the calculation of both $B\Diff_{D^3}(L)_0$ and $B \Diff_{\text{pt}}(L)_0$ we will exploit the covering of $\Isom(L)_0$ by $\SO(2)\times \SO(2)$ as discussed in Lemma \ref{lem: the descenting isometries}.

\subsection{The mapping class groups}
Our goal in this section is to calculate $\pi_0\Diff(M)$, the mapping class group of $M$.
\begin{lemma}\label{lem: descending differentials fixing points}
    Consider the inclusions
    \[\iota_{1j} \colon \SO(2)\hookrightarrow \Isom^+_{\{1j\}}(S^3)\]
    be the inclusion given as $e^{2ti} \mapsto F(e^{ti}, e^{-ti})$ and 
    \[\iota_{1}\colon \SO(2) \hookrightarrow \Isom^+_{\{1\}}(S^3)\]
    be the inclusion given as $e^{2ti} \mapsto F(e^{ti}, e^{ti})$ for all $t\in [0, \pi)$. Let $x$ denote either $1j$ or $1$ and $p^\ast\colon \Norm(C_m)_0\to \Diff_{p(x)}(L)_0$ the map induced by the projection $p\colon S^3\to L$ where $\Norm(C_m)$ is the normalizer of the $C_m < \Isom^+(S^3)$ that we are quotienting $S^3$ by to gain $p$. 
    Given an identification of the tangent space of at $x$ with $\mathbb{R}^3$, we get that the composition
    \[\SO(2)\overset{\iota_{x}}{\to} \Norm(C_m)_0 \overset{p^\ast}{\to}\Diff_{\{p(x)\}}(L)_0\overset{T_{x}}{\to}\GL^+_3(\mathbb{R})\]
    is the inclusion.
\end{lemma}
\begin{proof}
    Both of $\iota_1$ and $\iota_{1j}$ land in the $\SO(2)\times\SO(2) = F(S^1, S^1)$ subgroup of $\Isom^+(S^3)$ that is always in the normalizer of the subgroup we quotient by to get a generic lens space.
    The action of $C_m$ on $S^3$ is a free action of a finite discrete group, and therefore $\varepsilon$ chosen small enough, each point in $B_x(\varepsilon)$, where $B_{q_0 + q_1j}(\varepsilon) = \{z_0+z_1j\in S^3 \,|\, |z_0-q_0|^2+|z_1-q_1|^2 < \varepsilon\}$.
    Furthermore the image of $\iota_{x}$ leaves $x$ fixed and in fact also $B_x(\varepsilon)$ as for $\zeta, z \in \mathbb{C}$, $|\zeta ^2 z| = |z|$ and $F(\zeta, \zeta)$ is multiplication of the second coordinate by $\zeta^2$ and $F(\zeta, \zeta^{-1})$ is multiplication of the first coordinate by $\zeta^2$.
    By all this we really mean that we get a diagram as follows:
   \[\begin{tikzcd}
    	{B_x(\varepsilon)} && {B_x(\varepsilon)} \\
    	{p(B_x(\varepsilon))} && {p(B_x(\varepsilon)).}
    	\arrow["{\left.\iota_x(\zeta)\right|_{B_x(\varepsilon)}}", from=1-1, to=1-3]
    	\arrow["\cong"', from=1-1, to=2-1]
    	\arrow["\cong"', from=1-3, to=2-3]
    	\arrow["{\left.p\circ\iota_x(\zeta)\right|_{p(B_x(\varepsilon))}}", from=2-1, to=2-3]
    \end{tikzcd}\]
    Therefore choosing the charts on $L$ to be gained locally from charts on $S^3$ through $p$ we see that the differential of $p\circ\iota_x(\zeta)$ at $p(x)$ agrees with the differential of $\iota_x(\zeta)$ at $x$.

    The composition $T_{x}\circ \iota_{x}\colon \SO(2) \to \GL_3(\mathbb{R})$ becomes the inclusion, given by block summing with the one-by-one identity matrix (we restrict the differential of $\iota_x(A)$ which is block summing the matrix of $A$ with a two-by-two identity matrix to the space spanned by the other three standard basis vectors besides $x$).
\end{proof}
\begin{theorem}\label{thm: lens space diffs pi_0's}
    For a generic lens space $L$, the inclusions $\Diff_{\text{pt}}(L)\hookrightarrow \Diff(L)$ and $\Diff_{D^3}(L)\hookrightarrow \Diff_{\text{pt}}(L)$ induce isomorphisms on path components, and we have
    \[\pi_0(\Diff_{D^3}(L))\cong\pi_0(\Diff_{\text{pt}}(L))\cong \pi_0(\Diff(L))\cong C_2.\]
\end{theorem}
\begin{proof}
    The statement $\pi_0(\Diff(L))\cong C_2$ follows from the generalized Smale conjecture (Theorem \ref{thm: generalized smale conj}) and from $\Isom(L)\cong \Dih(S^1\tilde{\times}S^1)$ (quotienting $\Dih(S^1\tilde{\times}S^1)$ by $\langle F(\xi_{2m}^{q+1}), \xi_{2m}^{q-1})\rangle$ just results in an $m$-fold covering of $\Dih(S^1\tilde{\times}S^1)$ by itself).

    Let $1 = p(1)\in L$ for the quotient map $p\colon S^3\to L$.
    
    For $\pi_0(\Diff_{\text{pt}}(L))\cong \pi_0(\Diff(L))$ consider the fiber sequence
    \[\Diff_{\{1\}}(L)\to \Diff(L)\to L \cong \Emb(\text{pt}, L)\]
    this yields an exact sequence
    \[\pi_1(\Isom(L), \text{id}) \overset{f}{\to} \pi_1(L, 1)\to \pi_0(\Diff_{\{1\}}(L) )\overset{g}{\to} \pi_0(\Diff(L))\to \pi_0(L)\cong\text{pt}.\]
    To see that $g$ is an isomorphism we just need that $f$ is surjective.
    $\pi_1(L)$ is cyclic so all we have to show is that $f$ hits its generator.
    $p\circ \gamma$ generates $\pi_1(L)$ for $\gamma(t) = e^{\frac{2\pi i t}{m}}$ by covering theory, as $\xi_m = F(\xi_{2m}^{q+1}, \xi_{2m}^{q-1})(1)$, and $F(\xi_{2m}^{q+1}, \xi_{2m}^{q-1})$ is the generator of the $C_m$-action on $S^3$ we quotient by.
    Now we just have to see that $\gamma$ can be given by a path $\lambda$ in $\Norm(C_m) = \Dih(S^1\tilde{\times}S^1) = \langle F(S^1\times S^1), F(j, j) \rangle$ so that $\lambda(t)(1) = \gamma(t)$ and $\lambda$ becomes a loop in $\Isom(L)$.
    Such a path may be constructed as $\lambda(t) = f(\xi_{2m}^{t(q+1)}, \xi_{2m}^{t(q-1)})$, where $f(q_1, q_2)$ denotes the isometry of $L$ induced by $F(q_1, q_2)$ for any $q_1$ and $q_2$ this makes sense for.

    For $\pi_0(\Diff_{D^3}(L))\cong\pi_0(\Diff_{\text{pt}}(L))$ consider the homotopy fiber sequence
    \[\Diff_{D^3}(L) \to \Diff_{\{1\}}(L) \overset{T_1}{\to} \GL_3^{+}(\mathbb{R})\simeq SO(3).\]
    This gives rise to the exact sequence
    \[\pi_1(\Diff_{\{1\}}(L), \text{id}) \overset{f}{\to} \pi_{1}(\SO(3), \text{id})\to \pi_0(\Diff_{D^3}(L) )\overset{g}{\to} \pi_0(\Diff_{\{1\}}(L))\to \pi_0(\SO(3))\simeq \text{pt}.\]
    Again we have to see that $f$ is surjective.
    We have $\GL_3^{+}(\mathbb{R})\simeq \SO(3) \cong D^3/\sim$ where on $D^3$ we identify the antipodal points of $\partial D^3$, we take $D^3= \{x\in \mathbb{R}^3 \,|\, |x|\leq \pi\}$ and then each point $x\in D^3$ of it corresponds to the rotation around the span of $\{x\}$ in $\mathbb{R}^3$ by the angle $|x|$ and clockwise or counter clockwise depending on the sign of $x$, the origin corresponds to the identity.
    $\pi_1(\SO(3), \text{id}) = C_2$ generated by the loops given by $\gamma\colon [0, 1]\to D^3/\sim$, with $\gamma(t)= tx - (1-t)x$ for some $x\in \partial D^3$.
    This means that we want a loop $\lambda$ in $\Diff_{\{1\}}(L)$ with $T_1\lambda(t)$ being rotation by $(2t-1)\pi$ around some axis (as rotation by $\theta$ around an axis spanned by $x$ is rotation by $-\theta$ around the axis given by $-x$).
    Consider $\lambda(t)$ given by $F(\zeta_t, \zeta_t)$ for $\zeta_t = e^{\pi i t}$, since $\zeta_t\in S^1$, $F(\zeta_t, \zeta_t)(z_0+z_1j) = z_0+\zeta_t^2 z_1 j$.
    This is essentially the loop in $\Isom^+_1(S^3)$ given by $\iota_1(S^1)$ and therefore by Lemma \ref{lem: descending differentials fixing points} we conclude.
\end{proof}
Finally, we compute the path components of $\Diff(M, S)\simeq \Diff(M)$.
Before this calculation let us present a handy commutative diagram that will come up in another context later as well.
\begin{remark}\label{rem: handy commutative diagram}
    The following is a commutative diagram:
   \[\begin{tikzcd}[cramped,row sep=large]
    	{\Diff_{L_1\setminus \interior{D^3}}(M)} & {\Diff_\partial(L_2\setminus\interior{D^3})} & {\Diff_{D^3}(L_2)} \\
    	{\Diff(L_2\setminus \interior{D^3})} & {\Diff_{\text{pt}}(L_2, D^3)} & {\Diff_{\text{pt}}(L_2).}
    	\arrow["\simeq", from=1-1, to=1-2]
    	\arrow["{(\text{res}^M_{L_2\setminus \interior{D^3}})_\ast}", from=1-1, to=2-1]
    	\arrow[dashed, hook', from=1-2, to=2-1]
    	\arrow["\simeq"', from=1-3, to=1-2]
    	\arrow[dashed, hook', from=1-3, to=2-2]
    	\arrow[from=1-3, to=2-3]
    	\arrow["\simeq"', from=2-2, to=2-1]
    	\arrow["\simeq", from=2-2, to=2-3]
    \end{tikzcd}\]
\end{remark}
\begin{theorem}\label{thm: mapping class group}
    The mapping class group of $M\cong L_1\#L_2$ where $L_1$ and $L_2$ are non-diffeomorphic generic lens spaces is
    \[\pi_0 (\Diff(M)) \cong C_2\times C_2.\]
\end{theorem}
\begin{proof}
    We consider the commutative diagram, where both rows are fiber sequences:
    \[\begin{tikzcd}
	{\Diff_{L_1\setminus\interior{D^3}}(M)} & {\Diff(M, L_1\setminus\interior{D^3})} & {\Diff(L_1\setminus\interior{D^3})} \\
	{\Diff(L_2\setminus\interior{D^3})} & {\Diff(L_2\setminus\interior{D^3}) \times \Diff(L_1\setminus\interior{D^3})} & {\Diff(L_1\setminus\interior{D^3}).}
	\arrow[from=1-1, to=1-2]
	\arrow[from=1-1, to=2-1]
	\arrow[from=1-2, to=1-3]
	\arrow[from=1-2, to=2-2]
	\arrow[from=1-3, to=2-3]
	\arrow[from=2-1, to=2-2]
	\arrow[from=2-2, to=2-3]
\end{tikzcd}\]
    This induces a comparison of long exact sequences.
   \[\begin{tikzcd}[cramped,column sep=tiny]
    	{\pi_1\Diff(L_1\setminus\interior{D^3})} & {\pi_0\Diff_{L_1\setminus\interior{D^3}}(M)} & {\pi_0\Diff(M, L_1\setminus\interior{D^3})} & {\pi_0\Diff(L_1\setminus\interior{D^3})} \\
    	{\pi_1\Diff(L_1\setminus\interior{D^3})} & {\pi_0\Diff(L_2\setminus\interior{D^3})} & {\pi_0\Diff(L_2\setminus\interior{D^3}) \times \pi_0\Diff(L_1\setminus\interior{D^3})} & {\pi_0\Diff(L_1\setminus\interior{D^3}).}
    	\arrow["{\partial^\prime}", from=1-1, to=1-2]
    	\arrow[equal, from=1-1, to=2-1]
    	\arrow["{\iota_\ast}", from=1-2, to=1-3]
    	\arrow["{\left(\text{res}^M_{L_2\setminus\interior{D^3}}\right)_\ast}", from=1-2, to=2-2]
    	\arrow["{\left(\text{res}^M_{L_1\setminus\interior{D^3}}\right)_\ast}", from=1-3, to=1-4]
    	\arrow[from=1-3, to=2-3]
    	\arrow[equal, from=1-4, to=2-4]
    	\arrow["\partial", from=2-1, to=2-2]
    	\arrow[from=2-2, to=2-3]
    	\arrow[from=2-3, to=2-4]
    \end{tikzcd}\]
    We have that
    \[\pi_0\Diff_{L_1\setminus\interior{D^3}}(M)\cong \pi_0\Diff_{D^3}(L_2)\cong C_2\]
    and
    \[\pi_0\Diff(L_1\setminus\interior{D^3})\cong \pi_0\Diff_{\text{pt}}(L_1)\cong C_2.\]

    In the above diagram $\partial$ is $0$ by exactness, and $\left(\text{res}^M_{L_2\setminus\interior{D^3}}\right)_\ast$ is an isomorphism
    after considering the commutative diagram from Remark \ref{rem: handy commutative diagram} and Theorem \ref{thm: lens space diffs pi_0's}.
    This means that $\partial^\prime$ is $0$ by commutativity.
    Thus $\iota_\ast$ is injective.
    We furthermore have that $\left(\text{res}^M_{L_1\setminus\interior{D^3}}\right)_\ast$ is surjective by Lemma \ref{lem: extendability based on boundary}.
    Now we apply the 5-lemma to
    \[\begin{tikzcd}[column sep=large]
	0 & {C_2} & {\pi_0\Diff(M, L_1\setminus\interior{D^3})} & {C_2} & 0 \\
	0 & {C_2} & {C_2 \times C_2} & {C_2} & 0
	\arrow["{\partial^\prime}", from=1-1, to=1-2]
	\arrow[equal, from=1-1, to=2-1]
	\arrow["{\iota_\ast}", from=1-2, to=1-3]
	\arrow["\cong", from=1-2, to=2-2]
	\arrow["{\left(\text{res}^M_{L_1\setminus\interior{D^3}}\right)_\ast}", from=1-3, to=1-4]
	\arrow[from=1-3, to=2-3]
	\arrow[from=1-4, to=1-5]
	\arrow["\cong", from=1-4, to=2-4]
	\arrow[equal, from=1-5, to=2-5]
	\arrow["\partial", from=2-1, to=2-2]
	\arrow[from=2-2, to=2-3]
	\arrow[from=2-3, to=2-4]
	\arrow[from=2-4, to=2-5]
\end{tikzcd}\]
    and conclude that $\pi_0 \Diff(M)\cong \pi_0\Diff(M, L_1\setminus\interior{D^3})\cong C_2\times C_2$.
\end{proof}

\section{Computations on the identity path components}\label{the computation}
In this section $L$ will always denote a generic lens space.
We start with establishing some background and notation for the calculation.
\cite[Theorem 15.9]{miln74} implies that 
the rational cohomology ring $H^\ast(B\SO(n))$ is a polynomial ring over $\mathbb{Q}$ generated by
\begin{enumerate}
    \item in case $n$ is odd, the Pontryagin classes $p_1, \dots, p_{(n-1)/2}$
    \item in case $n$ is even, the Pontryagin classes $p_1, \dots, p_{n/2}$ and the Euler class $e$, where $e^2 = p_{n/2}$.
\end{enumerate}   
Here the degrees are as follows: $|p_k| = 4k$ and $|e| = n$.
The inclusion $\SO(n)\times\SO(m)\to \SO(n+m)$ given by block summing induces the Whitney sum on vector bundles, let us give two corollaries of this.

In $H^2(B\SO(2)\times B\SO(2))$ we will denote following the Künneth isomorphism $pr_1^\ast(e)$ as $e\otimes 1$ and $pr_2^\ast(e)$ as $1\otimes e$.
The map 
\[H^\ast(B\SO(4))\to H^\ast(B\SO(2)\times B\SO(2))\] 
induced by the inclusion of $\SO(2)\times \SO(2) \hookrightarrow \SO(4)$ sends $p_1$ to $(e\otimes 1)^2 + (1\otimes e)^2$ and $e$ to $(e\otimes 1)(1\otimes e)$.
Similarly the map 
\[H^\ast(B\SO(4))\to H^\ast(B\SO(3))\]
induced by block sum with the identity, sends $p_1$ to $p_1$ and $e$ to $0$.
\begin{lemma}\label{lem: preliminary s.seq. comparison}
    In the rational cohomological Leray-Serre spectral sequence of 
    \[S^3\to S^3\hq(\SO(2)\times\SO(2))\to B\SO(2)\times B\SO(2)\]
    the differential $d^4\colon E_4^{0, 3}\to E_4^{4, 0}$ sends the fundamental class of $S^3$ to a non-zero multiple of $(e\otimes 1)(1\otimes e)$.
\end{lemma}
\begin{proof}
    Applying Lemma \ref{lem: id path component homotopical orbit stabilizer} in light of Example \ref{eg: S^3 is SO(4) locally retractile} we have in particular $B\SO(3)\cong S^3\hq \SO(4)$ and under this homeomorphism $S^3\hq\SO(4)\to B\SO(4)$ becomes the map $B\SO(3)\hookrightarrow B\SO(4)$ induced by the inclusion $\SO(3)\hookrightarrow\SO(4)$ as $\SO(3)$ is the stabilizer subgroup of $1 + 0j\in S^3$.    
    
    We inspect the cohomological Leray-Serre spectral sequence of \[S^3\to S^3\hq\SO(4)\to B\SO(4).\]
    Note that the only non-zero differentials are on the $E_4$-page as $E_2^{p, q} \cong H^p(B\SO(4))\otimes H^q(S^3)$.
    Since 
    \[H^4(B\SO(4))\cong E_2^{4, 0}\rrightarrow E_\infty^{4, 0}\cong H^4(S^3\hq\SO(4))\]
    is induced by the map $S^3\hq\SO(4)\to B\SO(4)$ and 
    we conclude that $\image(d^4\colon E_4^{0, 3}\to E_4^{4, 0}) = \langle e\rangle$.

    Now the comparison 
\[\begin{tikzcd}[cramped]
	{S^3} & {S^3\hq\SO(4)} & {B\SO(4)} \\
	{S^3} & {S^3\hq(\SO(2)\times\SO(2))} & {B(\SO(2)\times\SO(2))}
	\arrow[from=1-1, to=1-2]
	\arrow[from=1-2, to=1-3]
	\arrow[shift left, no head, from=2-1, to=1-1]
	\arrow[no head, from=2-1, to=1-1]
	\arrow[from=2-1, to=2-2]
	\arrow[from=2-2, to=1-2]
	\arrow[from=2-2, to=2-3]
	\arrow["i"', from=2-3, to=1-3]
\end{tikzcd}\]
    induces a comparison of spectral sequences.
    We know that $i^\ast(e) = (e\otimes 1)(1\otimes e)$ and from this we conclude.
\end{proof}
\subsection{The diffeomorphisms fixing a point}
We want to compare $\Diff_{\text{pt}}(L)$ to $\Diff_{\text{pt}}^+(S^3)$, but not all of the diffeomorphisms of $S^3$ factor through the quotient, in fact similarly to Lemma \ref{lem: the descenting isometries} exactly those do which are in the normalizer of the $C_m$ subgroup of $\SO(4) = \Isom^+(S^3) < \Diff^+(S^3)$ that we mod out by.
This description gives us the following diagram:
\[\begin{tikzcd}
	{\Diff^{+}(S^3)} & {\Norm_{\Diff^+(S^3)}(C_m)_0} & {\Diff(L)_0} \\
	{\SO(4)} & {\SO(2)\times\SO(2)} & {\Isom(L)_0} \\
	{S^3}\arrow[loop above, out=120, in=70, distance=15] & {S^3}\arrow[loop above, out=120, in=70, distance=15] &  L.\arrow[loop above, out=120, in=70, distance=15]
	\arrow[from=1-2, to=1-1]
	\arrow[from=1-2, to=1-3]
	\arrow["\simeq"', hook, from=2-1, to=1-1]
	\arrow[hook, from=2-2, to=1-2]
	\arrow[from=2-2, to=2-1]
	\arrow["{\sim_\mathbb{Q}}", from=2-2, to=2-3]
	\arrow["\simeq", hook, from=2-3, to=1-3]
	\arrow[equal, from=3-2, to=3-1]
	\arrow["{\sim_\mathbb{Q}}", from=3-2, to=3-3]
\end{tikzcd}\]
\begin{notation}
    By $\sim_\mathbb{Q}$ we denote that the given map induces isomorphism on rational cohomology.
\end{notation}
In this case the maps indicated to induce isomorphisms on rational cohomology do so by virtue of the fact that the maps $F(S^1, S^1) = \SO(2)\times\SO(2)\to\Norm(C_m)_0 = \Dih(S^1\tilde{\times}S^1)_0$ and $S^3\to L$ in the diagram are m-fold coverings.

By naturality we get a zig-zag of homotopy fiber sequences
\begin{equation}\label{eq: emb of a point comparison}
\begin{tikzcd}
	{S^3} & {S^3\hq \SO(4)} & {B\SO(4)} \\
	{S^3} & {S^3\hq (\SO(2)\times \SO(2))} & {B(\SO(2)\times\SO(2))} \\
	L & {L\hq \Isom(L)_0} & {B\Isom(L)_0.}
	\arrow[from=1-1, to=1-2]
	\arrow[from=1-2, to=1-3]
	\arrow[equal, from=2-1, to=1-1]
	\arrow[from=2-1, to=2-2]
	\arrow["{\sim_\mathbb{Q}}", from=2-1, to=3-1]
	\arrow[from=2-2, to=1-2]
	\arrow[from=2-2, to=2-3]
	\arrow[from=2-2, to=3-2]
	\arrow[from=2-3, to=1-3]
	\arrow["{\sim_\mathbb{Q}}", from=2-3, to=3-3]
	\arrow[from=3-1, to=3-2]
	\arrow[from=3-2, to=3-3]
\end{tikzcd}
\end{equation}
Here the middle map of the bottom comparison is also a rational cohomology isomorphism by the naturality properties of the Leray-Serre spectral sequences, see \cite[Proposition 5.13]{HatchSSeq}.

\begin{theorem}\label{thm: rat cohom of diff(generic lens space) fixed a point}
    For a generic lens space $L$,
    \[H^\ast(B\Diff_{\text{pt}}(L)_0)\cong \mathbb{Q}[\mu, \eta]/( \mu\eta)\]
    where $|\mu|=|\eta| = 2$.
    Furthermore there is a surjection of graded algebras
    \[H^\ast(B\SO(2)\times B\SO(2)) \rrightarrow H^\ast(B\Diff_{\text{pt}}(L)_0)\]
    induced by the zig-zag $B\SO(2)\times B\SO(2) \overset{\sim_\mathbb{Q}}{\to} B\Isom(L)_0 \leftarrow L\hq\Isom(L)_0 \simeq B\Diff_{\text{pt}}(L)_0$,
    sending the pullbacks $1\otimes e$ and $e\otimes 1$ of the Euler class $e\in H^\ast(B\SO(2))$ along the two projections to $\mu$ and $\eta$.
\end{theorem}
\begin{proof}
    By Theorem \ref{Emb is locally retractile}, $\Emb(\text{pt}, L)\cong L$ is $\Diff(L)$-locally retractile.
    Lemma \ref{local retractileness} (3) and (4) implies that it is also $\Diff(L)_0$-locally retractile and that the $\Diff(L)_0$ action on $L$ is transitive.
    Lemma \ref{lem: id path component homotopical orbit stabilizer} and Theorem \ref{thm: lens space diffs pi_0's} implies that $\Diff_\text{pt}(L)_0\simeq \Emb(\text{pt}, L)\hq \Diff(L)_0$.
    Finally, by Theorem \ref{thm: generalized smale conj} we have
    \[L\hq \Isom(L)_0 \simeq B\Diff_{\text{pt}}(L)_0.\]
    
    By the comparison (\ref{eq: emb of a point comparison}) we reduce to computing $H^\ast(S^3\hq(\SO(2)\times\SO(2)))$.
    Using Lemma \ref{lem: preliminary s.seq. comparison} and the fact that the only non-zero differentials in the cohomological Leray Serre spectral sequence of 
    \[S^3\to S^3\hq(\SO(2)\times \SO(2))\to B\SO(2)\times B\SO(2)\]
    are on the $E_4$-page, we conclude that the spectral sequence collapses on the $E_5$-page, and examining the cup product structure that the $d_4$ differentials hit everything in the ideal $((e\otimes 1)(1\otimes e))$ and leave only the zeroth row to be non-zero in $E_\infty$.
\end{proof}
\subsection{The diffeomorphisms fixing a disc}
Similarly to before we use the diagram
\[\begin{tikzcd}
	{\SO(4)} & {\SO(2)\times\SO(2)} & {\Isom(L)_0} \\
	{\Emb^{+}(D^3, S^3)}\arrow[loop above, out=120, in=70, distance=15] & {\Emb^{+}(D^3, S^3)}\arrow[loop above, out=120, in=70, distance=15] & \Emb^{+}(D^3, L).\arrow[loop above, out=120, in=70, distance=15]
	\arrow[from=1-2, to=1-1]
	\arrow["{\sim_\mathbb{Q}}", from=1-2, to=1-3]
	\arrow[equal, from=2-2, to=2-1]
	\arrow["{\sim_\mathbb{Q}}", from=2-2, to=2-3]
\end{tikzcd}\]
This diagram implies by naturality that we have a zig-zag of fiber sequences as follows:
\begin{equation}\label{eq: second fib seq comparison}
\begin{tikzcd}[cramped,column sep=small]
	{\Emb^{+}(D^3, S^3)} & {\Emb^{+}(D^3, S^3)\hq \SO(4)} & {B\SO(4)} \\
	{\Emb^{+}(D^3, S^3)} & {\Emb^{+}(D^3, S^3)\hq (\SO(2)\times \SO(2))} & {B(\SO(2)\times\SO(2))} \\
	\Emb^{+}(D^3, L) & {\Emb^{+}(D^3, L)\hq \Isom(L)_0} & {B\Isom(L)_0.}
	\arrow[from=1-1, to=1-2]
	\arrow[from=1-2, to=1-3]
	\arrow[equal, from=2-1, to=1-1]
	\arrow[from=2-1, to=2-2]
	\arrow["{\sim_\mathbb{Q}}", from=2-1, to=3-1]
	\arrow[from=2-2, to=1-2]
	\arrow[from=2-2, to=2-3]
	\arrow[from=2-2, to=3-2]
	\arrow[from=2-3, to=1-3]
	\arrow["{\sim_\mathbb{Q}}", from=2-3, to=3-3]
	\arrow[from=3-1, to=3-2]
	\arrow[from=3-2, to=3-3]
\end{tikzcd}    
\end{equation}

\begin{theorem}\label{thm: rat cohom of diff(generic lens space) fixed a disc}
    For a generic lens space $L$,
    \[H^\ast(B\Diff_{D^3}(L)_0)\cong \mathbb{Q}[\mu, \eta]/( \mu^2+\eta^2, \mu\eta)\]
    where $|\mu|=|\eta| = 2$.
    Furthermore there is a surjection of graded algebras
    \[H^\ast(B\SO(2)\times B\SO(2)) \rrightarrow H^\ast(B\Diff_{D^3}(L)_0)\]
    induced by the zig-zag $B(\SO(2)\times \SO(2))\overset{\sim_\mathbb{Q}}{\to}B\Isom(L)_0\leftarrow \Emb^+(D^3, L)\hq \Isom(L)_0$
    sending the pullbacks $1\otimes e$ and $e\otimes 1$ of the Euler class $e\in H^\ast(B\SO(2))$ along the two projections to $\mu$ and $\eta$.
\end{theorem}
\begin{proof}
    $L$ is parallelizable, meaning $\Fr^+(L)\cong L\times \GL_3^+(\mathbb{R})\simeq L\times \SO(3)$, because it is a closed orientable 3-manifold (see \cite{bened18}).  
    Thus Theorem \ref{embeddings of discs are framings} implies $\Emb^+(D^3, L)\simeq L\times \SO(3)$.
    This means it is path connected, which is instrumental in using the homotopy orbit stabilizer lemma. 
    
    By Theorem \ref{Emb is locally retractile}, $\Emb(D^3, L)\cong L$ is $\Diff(L)$-locally retractile.
    Lemma \ref{local retractileness} (3) and (4) implies that it is also $\Diff(L)_0$-locally retractile and that the $\Diff(L)_0$ action on $L$ is transitive.
    Lemma \ref{lem: id path component homotopical orbit stabilizer} and Theorem \ref{thm: lens space diffs pi_0's} implies that $\Diff_{D^3}(L)_0\simeq \Emb(D^3, L)\hq \Diff(L)_0$.
    Finally, by Theorem \ref{thm: generalized smale conj} we have
    \[\Emb^+(D^3, L)\hq \Isom(L)_0\simeq B\Diff_{D^3}(L)_0.\]
    Similar argument shows
    \[\Emb^+(D^3, S^3)\hq\SO(4)\simeq B\Diff_{D^3}(S^3)\simeq \text{pt}.\]
    By Theorem \ref{embeddings of discs are framings} we also have that $\Emb^+(D^3, S^3)\simeq S^3\times \SO(3)$.
    Inspecting (\ref{eq: second fib seq comparison}) we can see that again we may reduce to computing \[H^\ast(\Emb^+(D^3, S^3)\hq(\SO(2)\times\SO(2))).\]
    Let us denote $E_\bullet^{\bullet, \bullet}$ the cohomological Leray Serre spectral sequence associated to 
    \[\Emb^+(D^3, S^3)\to \Emb^+(D^3, S^3)\hq\SO(4)\to B\SO(4).\]
    Let us denote $D_\bullet^{\bullet, \bullet}$ the cohomological Leray Serre spectral sequence associated to 
    \[\Emb^+(D^3, S^3)\to \Emb^+(D^3, S^3)\hq(\SO(2)\times\SO(2))\to B\SO(2)\times B\SO(2).\]
    Note that $E_2^{p, q}\cong E_2^{p, 0}\otimes E_2^{0, q}$ and also $D_2^{p, q}\cong D_2^{p, 0}\otimes D_2^{0, q}$.
    Let us use the notation
    \[H^\ast(\Emb^{+}(D^3, S^3))\cong H^\ast(S^3)\otimes_\mathbb{Q} H^\ast(\SO(3), \mathbb{Q})\cong \mathbb{Q}[\alpha, \beta]/\langle \alpha^2, \beta^2\rangle\]
    and
    $\mu = e\otimes 1$, $\eta = 1\otimes e\in H^2(B\SO(2)\times B\SO(2))$.
    With these notations the comparison of the fiber sequences $E_\bullet^{\bullet, \bullet}$ and $D_\bullet^{\bullet, \bullet}$ is laid out in Figure \ref{fig:sseqs2}, where the dots denote non-zero vector spaces that have too many generators to list.

    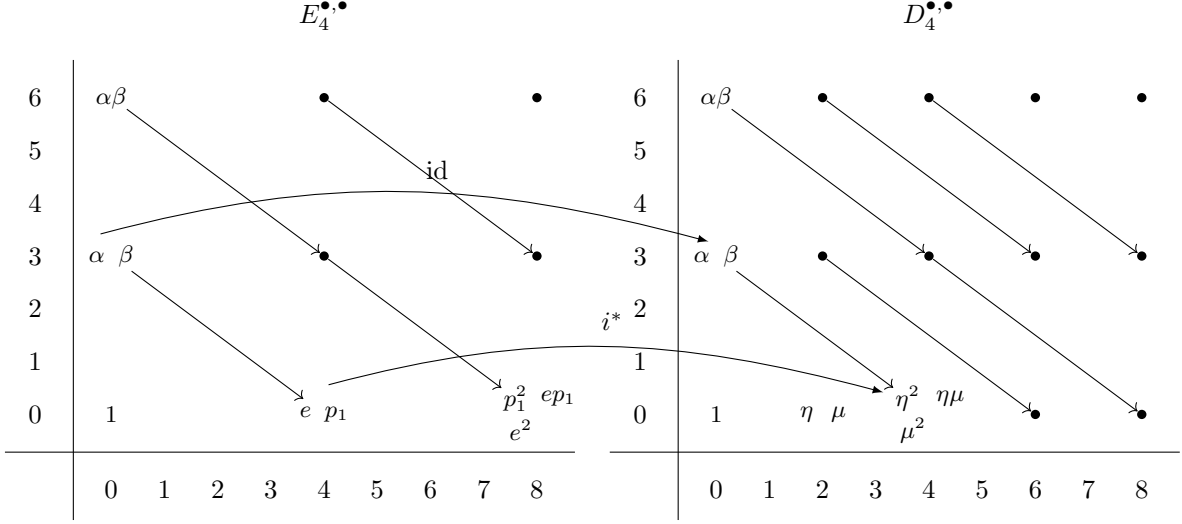
\begin{figure}[ht]
    \advance\leftskip-1cm
    \caption{Comparing spectral sequences}
    \begin{sseqpage}[title = $E_4^{\bullet, \bullet}$,
    cohomological Serre grading,
    class pattern = myclasspattern,
    classes = { draw = none },
    class labels = { font = \small },
    xscale = 0.7,
    yscale = 0.7]
    \class["1"](0,0)
    \class["e\;\;p_1"](4, 0)
    \class["p_1^2"](8, 0)
    \class["e^2"](8, 0)
    \class["e p_1"](8, 0)
    \class["\alpha\;\;\beta"](0, 3)
    \class["\alpha \beta"](0, 6)
    \class[{ black, fill }](4, 3)
    \class[{ black, fill }](4, 6)
    \class[{ black, fill }](8, 3)
    \class[{ black, fill }](8, 6)
    \d4(0,3) 
    \d4(0, 6)
    \d4(4, 3)
    \d4(4, 6)
    \end{sseqpage}
    \quad
    \begin{sseqpage}[title = $D_4^{\bullet, \bullet}$,
    cohomological Serre grading,
    class pattern = myclasspattern,
    classes = { draw = none },
    class labels = { font = \small },
    xscale = 0.7,
    yscale = 0.7]
    \class["1"](0,0)
    \class["\eta"](2, 0) \class["\mu"](2, 0)
    \class["\eta^2"](4, 0)
    \class["\mu^2"](4, 0)
    \class["\eta \mu"](4, 0)
    \class[{ black, fill }](6, 0)
    \class[{ black, fill }](8, 0)
    \class["\alpha\;\;\beta"](0, 3)
    \class["\alpha \beta"](0, 6)
    \class[{ black, fill }](2, 3)
    \class[ { black, fill }](2, 6)
    \class[ { black, fill }](4, 6)
    \class[ { black, fill }](4, 3)
    \class[ { black, fill }](6, 6)
    \class[ { black, fill }](6, 3)
    \class[ { black, fill }](8, 6)
    \class[ { black, fill }](8, 3)
    \d4(0,3) 
    \d4(0, 6)
    \d4(2,3) 
    \d4(2, 6)
    \d4(4, 3)
    \d4(4, 6)
    \end{sseqpage}
    \begin{tikzpicture}[overlay, remember picture]
    \draw[-latex] (-14.3, 3.8) to[out=15,in=165] (-6.3, 3.7)  node [above left = 0.7 and 3.3] {$\text{id}$};
    \draw[-latex] (-11.3, 1.8) to[out=15,in=165] (-4, 1.7)
    node [above left = 0.7 and 3.3] {$i^\ast$};
    \end{tikzpicture}
    \label{fig:sseqs2}
    \end{figure}

    Firstly, we want that $\langle\prescript{E}{}d_4^{0, 3}(\alpha)\rangle=\langle e \rangle$.
    To see this we use a comparison of spectral sequences through the following diagram:
    \[\begin{tikzcd}
        {S^3} & {S^3\hq \SO(4)} & {B\SO(4)}\\
	{S^3\times \SO(3)} & {(S^3\times \SO(3))\hq \SO(4)} & {B\SO(4).}
	\arrow[from=1-1, to=1-2]
	\arrow[from=1-2, to=1-3]
	\arrow[from=2-1, to=1-1]
	\arrow[from=2-1, to=2-2]
	\arrow[from=2-2, to=1-2]
	\arrow[from=2-2, to=2-3]
	\arrow["\simeq", from=2-3, to=1-3]
\end{tikzcd}\]
    Where the map $S^3\times \SO(3)\to S^3$ is the projection onto the first coordinate.
    This is because the weak equivalence $\Emb^+(D^3, S^3)\simeq S^3\times \SO(3)$ records the point the origin is sent to, and the differential at the origin orthogonalized via Gram-Schmidt.    
    Therefore under this weak equivalence, the map $\Emb^+(D^3, S^3)\to \Emb(\text{pt}, S^3)$ induced by the inclusion of the origin into $D^3$, becomes the projection.

    This means that $\alpha\in H^\ast(S^3)$ is sent to $\alpha = pr_1^\ast(\alpha)\in H^\ast(S^3\times \SO(3))$ by the map induced on cohomology by the comparison of the fibers, and thus by Lemma \ref{lem: preliminary s.seq. comparison} we see that indeed $\langle\prescript{E}{}d_4^{0, 3}(\alpha)\rangle=\langle e \rangle$.

    Since $E_\infty^{\ast, \ast} \cong 0$ and the only non-trivial differentials in $E_\bullet^{\bullet, \bullet}$ are on the $E_4$-page, we have to have that $\langle\prescript{E}{}d_4^{0, 3}(\beta)\rangle=\langle p_1 \rangle$.
    We can see that the comparison yields
    \[\langle\prescript{D}{}d_4^{0, 3}(\alpha)\rangle=\langle \mu\eta \rangle\]
    and
    \[\langle\prescript{D}{}d_4^{0, 3}(\beta)\rangle=\langle \mu^2+\eta^2 \rangle.\]
    We have
    \[\dim(E_2^{2k, 6}) + \dim(E_2^{2k+8, 0}) = \dim(E_2^{2k+4, 3})\]
    and $\dim(E_2^{6, 0}) = \dim(E_2^{2, 3})$.
    Furthermore inspecting the multiplicative structure we find that $\prescript{D}{}d_4^{2k, 6}\colon D_4^{2k, 6}\to D_4^{2k+4, 3}$ sends the generators of $D_4^{2k, 6}$ to an independent set in $D_4^{2k+4, 3}$ and that all the generators of $D_4^{2k+6, 0}$ are hit by $\prescript{D}{}d_4^{2k+2, 3}$ for all $k\geq 0$.
    This means that in fact in the $E_\infty$-page, the only non-trivial entries that remain are $D_\infty^{0, 0}$, $D_\infty^{2, 0}$, and $D_\infty^{4, 0}$.
    From this we conclude.
\end{proof}

\subsection{The whole identity path component}
To calculate $H^k(B\Diff(M)_0)$, we just have to run
the cohomological Leray-Serre spectral sequence of
\[B\Diff_{L_2\setminus\interior{D^3}}(M)_0\to B\Diff(M, L_2\setminus\interior{D^3})_0\to B\Diff(L_2\setminus\interior{D^3})_0.\]
Here the base is weakly equivalent to $B\Diff_{\text{pt}}(L_2)_0$ and the fiber is weakly equivalent to $B\Diff_{D^3}(L_1)$.
Let us recall our results from Theorem \ref{thm: rat cohom of diff(generic lens space) fixed a point} and Theorem \ref{thm: rat cohom of diff(generic lens space) fixed a disc}.
\[H^k(B\Diff_{D^3}(L)_0)\cong \begin{cases}
    \mathbb{Q}\text{ if } k= 0\\
    \mathbb{Q}\langle \mu,  \eta\rangle \text{ if } k = 2\\
    \mathbb{Q}\langle \mu^2\rangle \text{ if } k = 4\\
    0\text{ otherwise}
\end{cases}\]
Where the cup product structure is so that $\mu^2 = -\eta^2$ and $\mu\eta = 0$.
\[H^k(B\Diff_{\text{pt}}(L)_0)\cong \begin{cases}
    \mathbb{Q}\text{ if } k= 0\\
    \mathbb{Q}\langle \mu^{k/2},  \eta^{k/2} \rangle\text{ if } k \text{ is even}\\
    0\text{ otherwise}
\end{cases}\]
These imply that the $E_2$-page we are interested in looks as follows:
\[\begin{sseqpage}[title = The main spectral sequence,
    cohomological Serre grading,
    class pattern = myclasspattern,
    classes = { draw = none },
    class labels = { font = \small },
    xscale = 0.7,
    yscale = 0.7]
    \class["\mathbb{Q}"](0,0)
    \class["\mathbb{Q}^2"](0, 2)
    \class["\mathbb{Q}"](0, 4)
    \class["\mathbb{Q}^2"](2, 0)
    \class["\mathbb{Q}^2"](4, 0)
    \class["\mathbb{Q}^2"](6, 0)
    \class["\mathbb{Q}^4"](2, 2)
    \class["\mathbb{Q}^2"](2, 4)
    \class["\mathbb{Q}^2"](4, 4)
    \class["\mathbb{Q}^2"](6, 4)
    \class["\mathbb{Q}^4"](4, 2)
    \class["\mathbb{Q}^4"](6, 2)
    \class["\dots"](8, 2)
    \class["\dots"](8, 0)
    \class["\dots"](8, 4)
\end{sseqpage}\]
but since we are in even cohomological grading this collapses on the $E_2$-page and therefore we get that
\[H^n(B\Diff(M)_0)\cong \bigoplus_{k+l = n}H^k(B\Diff_{L_2\setminus\interior{D^3}}(M)_0)\otimes_{\mathbb{Q}} H^l(B\Diff(L_2\setminus\interior{D^3})_0).\]
\begin{theorem}\label{thm: main result}
    Let $L_1$ and $L_2$ be generic 3-dimensional lens spaces that are not diffeomorphic to each other, and $M \cong L_1\#L_2$.
    \[H^k(B\Diff(M)_0)\cong \begin{cases}
        \mathbb{Q} \;\,\text{ if }  k = 0\\
        \mathbb{Q}^4 \text{ if }  k = 2\\
        \mathbb{Q}^7 \text{ if } k = 4\\
        \mathbb{Q}^8 \text{ if $k$ is even and }\geq 6 \\
        0\text{ otherwise}
    \end{cases}\]
\end{theorem}

Now we will give a more information about the cup product structure:
Figure \ref{fig:main ho fib seq comp} shows a comparison that we also used in the proof of Theorem \ref{thm: mapping class group}.
\begin{figure}[ht]
    \caption{Comparing the homotopy fiber sequences}
    \[\begin{tikzcd}
	{B\Diff_{L_1\setminus\interior{D^3}}(M)_0} & {B\Diff(M, L_1\setminus\interior{D^3})_0} & {B\Diff(L_1\setminus\interior{D^3})_0} \\
	{B\Diff(L_2\setminus\interior{D^3})_0} & {B\Diff(L_2\setminus\interior{D^3})_0 \times B\Diff(L_1\setminus\interior{D^3})_0} & {B\Diff(L_1\setminus\interior{D^3})_0}
	\arrow[from=1-1, to=1-2]
	\arrow[from=1-1, to=2-1]
	\arrow[from=1-2, to=1-3]
	\arrow[from=1-2, to=2-2]
	\arrow[equal, from=1-3, to=2-3]
	\arrow[from=2-1, to=2-2]
	\arrow[from=2-2, to=2-3]
\end{tikzcd}\]
\label{fig:main ho fib seq comp}
\end{figure}
From it we get a comparison of the induced cohomological Leray-Serre spectral sequences.
The map $B\Diff_{L_1\setminus\interior{D^3}}(M)_0 \to B\Diff(L_2\setminus\interior{D^3})_0$ corresponds to $B\Diff_{D^3}(L_2)_0\to B\Diff_{\text{pt}}(L_1)_0$ under the commutative diagram from Remark \ref{rem: handy commutative diagram}.
As a consequence of Theorem \ref{thm: rat cohom of diff(generic lens space) fixed a point} and  Theorem \ref{thm: rat cohom of diff(generic lens space) fixed a disc} we have the following:
\begin{corollary}\label{lem: surj on cohom of fiber}
    The map induced by the inclusion $B\Diff_{D^3}(L_2)_0\to B\Diff_{\text{pt}}(L_1)_0$ induces a surjection on rational cohomology.
\end{corollary}
\begin{proof}
    There is a commutative diagram as follows:
    \[\begin{tikzcd}[cramped,column sep=tiny]
    	&&& {B\SO(2)\times B\SO(2)} \\
    	{B\Diff_{D^3}(L_2)_0} & {\Emb^+(D^3, L_2)\hq\Diff(L_2)_0} & {\Emb^+(D^3, L_2)\hq\Isom(L_2)_0} & {B\Isom(L_2)_0} \\
    	{B\Diff_{\text{pt}}(L_2)_0} & {\Emb(pt, L_2)\hq\Diff(L_2)_0} & {\Emb(\text{pt}, L_2)\hq\Isom(L_2)_0} & {B\Isom(L_2)_0.}
    	\arrow["{\sim_\mathbb{Q}}", from=1-4, to=2-4]
    	\arrow[from=2-1, to=3-1]
    	\arrow["\simeq", from=2-1, to=2-2]
    	\arrow["\simeq"', from=2-3, to=2-2]
    	\arrow[from=2-2, to=3-2]
    	\arrow[from=2-3, to=2-4]
    	\arrow[from=2-3, to=3-3]
    	\arrow[equal, from=2-4, to=3-4]
    	\arrow["\simeq", from=3-1, to=3-2]
    	\arrow["\simeq"', from=3-3, to=3-2]
    	\arrow[from=3-3, to=3-4]
    \end{tikzcd}\]
    Applying rational cohomology to this we obtain by Theorem \ref{thm: rat cohom of diff(generic lens space) fixed a point} and  Theorem \ref{thm: rat cohom of diff(generic lens space) fixed a disc} the commutativity of the following triangle:
    \[\begin{tikzcd}
	{H^\ast(B\Diff_{D^3}(L)_0)} & {H^\ast(B\SO(2)\times B\SO(2))} \\
	{H^\ast(B\Diff_{\text{pt}}(L)_0).}
	\arrow[two heads, from=1-2, to=1-1]
	\arrow[two heads, from=1-2, to=2-1]
	\arrow[from=2-1, to=1-1]
\end{tikzcd}\]
    This then shows that
    \[ H^\ast(B\Diff_{\text{pt}}(L_2)_0)\to H^\ast(B\Diff_{D^3}(L_2)_0)\]
    is surjective.
\end{proof}

The following lemma will be a core part of our argument for Theorem \ref{thm: the rational cohommology of the main gorups identitiy component}.
\begin{lemma}\label{lem: differential map on group cohomology}
    Let $L$ be a generic lens space and consider the map given by taking the differential at a point $\text{pt}$:
    \[T_{\text{pt}}\colon \Diff_{\text{pt}}(L) \to \GL^+_3(\mathbb{R}).\]

    On rational cohomology this map induces the map
    \[H^\ast(B GL^+_3(\mathbb{R}))\cong \mathbb{Q}[p_1] \to H^\ast(B\Diff_{\text{pt}}(L))\cong\mathbb{Q}[\mu, \eta]/(\mu\eta)\] that sends $p_1$ to $\mu^2+\eta^2$ with the usual notation for the cohomology of $\Diff_{\text{pt}}(L)$, and where we use $GL^+_3(\mathbb{R})\simeq \SO(3)$ and $p_1$ denotes the Pontryagin class.
\end{lemma}
\begin{proof}
    Let us use the same notations $\iota_{1}$ and $\iota_{1j}$ as in Lemma \ref{lem: descending differentials fixing points}.
    When thinking of $S^3$ as the unit sphere in $\mathbb{C}^2$, the image of $\iota_1$ consists of all the rotations of the first coordinate leaving the second coordinate fixed, the image of $\iota_{1j}$ consists of all the rotations of the second coordinate leaving the first coordinate fixed.
    
    This means that these maps factor through the quotient $\pi\colon S^3\to L$, meaning that we can get dashed maps, where $\pi^\ast\colon \Norm_{\Isom^+_{x}(S^3)}(C_m)_0\to\Diff_{\pi(x)}(L)_0$ denotes the map given by postcomposition with $\pi$.
    \begin{equation}\label{eq: iota pi business}
    \begin{tikzcd}[cramped]
    	{\{x\}\hq \SO(2)} && {\{\pi(x)\}\hq\Diff_{\{\pi(x)\}}(L)_0} \\
    	{S^3\hq(\SO(2)\times\SO(2))} & {L\hq\Isom(L)_0} & {L\hq\Diff(L)_0}
    	\arrow["B(\pi^\ast\circ\iota_{x})", dashed, from=1-1, to=1-3]
    	\arrow[from=1-1, to=2-1]
    	\arrow["\simeq"', from=1-3, to=2-3]
    	\arrow["\sim_{\mathbb{Q}}", from=2-1, to=2-2]
    	\arrow["\simeq", from=2-2, to=2-3]
    \end{tikzcd}
    \end{equation}
    Where $\pi\colon S^3 \to L$ is the quotient map, $x$ is either $1j$ or $1$, and in the diagram the left vertical map is induced by the inclusion
    \[\begin{tikzcd}[cramped, row sep=small, column sep=large]
    	{\SO(2)} & \SO(2)\times\SO(2) \\
    	\{x\} \arrow[loop above, out=120, in=70, distance=15] & S^3. \arrow[loop above, out=120, in=70, distance=15]
    	\arrow[hook, from=1-1, to=1-2, "\iota_{x}"]
    	\arrow[hook, from=2-1, to=2-2]
    \end{tikzcd}\]
    Let us investigate what this diagram looks like after taking rational cohomology.
    First, we must consider what this left vertical map induces in rational cohomology and for that we can use the commutative triangle
    \[\begin{tikzcd}[cramped]
    	{B\SO(2)} & {S^3\hq(\SO(2)\times \SO(2))} \\
    	& {B(\SO(2)\times\SO(2)).}
    	\arrow[from=1-1, to=1-2]
    	\arrow["{B\iota_x}"', from=1-1, to=2-2]
    	\arrow[from=1-2, to=2-2]
    \end{tikzcd}\]
    The vertical map in this triangle by Lemma \ref{lem: preliminary s.seq. comparison} on cohomology induces the quotient map $\mathbb{Q}[e\otimes 1, 1\otimes e] \to \mathbb{Q}[e\otimes 1, 1\otimes e]/((e\otimes 1)(1\otimes e))$.
    
    Furthermore since $\iota_\text{1}$ is a section of $\text{pr}_2$ but $\text{pr}_1\circ\iota_1$ is constant and $\iota_{1j}$ is a section of $\text{pr}_1$ but $\text{pr}_2\circ\iota_{1j}$ is constant, we have that $B\iota_1(e\otimes 1) = 0$ and $B\iota_1(1\otimes e) = e$ and 
    $B\iota_{1j}(e\otimes 1) = 0$ and $B\iota_{1j}(1\otimes e) = e$.
    
    Now we can conclude that applying cohomology to (\ref{eq: iota pi business}) we see that $\mu$ and $\eta$ are defined to be the preimages of $e\otimes 1$ and $1\otimes e$ respectively through the function that is induced by $S^3\hq (\SO(2)\times \SO(2))\to L\hq \Isom(L)_0$ on cohomology.    
    Now we have described all the maps except the dashed map in (\ref{eq: iota pi business}) but commutativity allows us to conclude that $B(\pi^\ast\circ\iota_1)$ sends $\mu$ to $0$ and $\eta$ to $e$,  and $B(\pi^\ast\circ\iota_{1j})$ sends $\mu$ to $e$ and $\eta$ to $0$.

    Furthermore as we have seen in Lemma \ref{lem: descending differentials fixing points} the following composition is still the same inclusion of $\SO(2)$:
    \[\SO(2)\overset{\iota_{x}}{\to} \Norm(C_m)_0\overset{\pi^\ast}{\to}\Diff_{\{\pi(x)\}}(L)_0\overset{T_{x}}{\to}\GL^+_3(\mathbb{R}).\]
    This means that 
    \[B(T_x\circ\pi^\ast\circ\iota_x)\colon B\SO(2)\to B\GL^+_3(\mathbb{R})\]
    on cohomology induces the map that sends $p_1$ to $e^2$ by the theory of characteristic classes ($e^2 = p_1$ in $H^\ast(B\SO(2))\cong \mathbb{Q}[e]$).
    
    Now we are almost ready to conclude, but first we have to relate the two maps $T_{1}$ and $T_{1j}$.
    The subgroups $\Diff_{\{1\}}(L)_0, \Diff_{\{1j\}}(L)_0 < \Diff(L)_0$ are conjugate to each other and we wish to exploit this fact.
    Let us fix a diffeomorphism $\psi\in \Diff(L)_0$ that is so that $\psi(1) = 1j$, note that existence of such $\psi$ follows from Lemma \ref{local retractileness} (3) and (4) and the fact that $L\cong\Emb(pt, L)$ is $\Diff(L)$-locally retractile.
    Conjugating some $\varphi\in \Diff_{\{1\}}(L)_0$ with $\psi$ we get $\psi\circ\varphi\circ\psi^{-1}\in \Diff_{\{1j\}}(L)_0$, let us denote $c_\psi$ the map sending $\varphi$ to $\psi\circ\varphi\circ\psi^{-1}$.
    When we think of $T_{x}$ as taking values in $\GL_3^+(\mathbb{R})$, we are identifying $T_{x}L$ with $\mathbb{R}^3$ via a chart, let us denote this chart by $\sigma_{x}$.
    We may assume that on a small neighborhood of $0$ the diffeomorphism $\sigma_{1j}\circ\psi\circ\sigma_{1}$ is the identity, this means that $T_1 = c_{\psi}\circ T_{1j}$.
    It is a general fact that an inner homomorphism of $G$ induces on $BG$ a map homotopic to the identity (however in general not based homotopic), see for example \cite[Chapter II Theorem 1.9]{adem13} but it also follows in our case directly from $\Diff(L)_0$ being path connected.
    The inclusion $\Diff_{x}(L)_0\hookrightarrow\Diff(L_0)$ induces furthermore a surjection on rational cohomology, $\mathbb{Q}[\mu, \eta] \to \mathbb{Q}[\mu, \eta]/(\mu\eta)$, and this means that $(Bc_{\psi})^\ast\colon H^\ast(B\Diff_{1j}(L)_0)\to H^\ast(B\Diff_{1}(L)_0)$ is the identity.
    We will identify these cohomology groups via this comparison.

    To conclude consider that 
    \[B(T_1\circ\pi^\ast\circ\iota_1) = B(T_{1j}\circ\pi^\ast\circ\iota_{1j}).\]
    These send $p_1$ to $e^2$.
    Furthermore $B(\pi^\ast\circ\iota_1)$ sends $\mu$ to $0$ and $\eta$ to $e$,  and $B(\pi^\ast\circ\iota_{1j})$ sends $\mu$ to $e$ and $\eta$ to $0$.
    This means that necessarily $T_1(p_1) = \eta^2 + a\mu^2$ and $T_{1j}(p_1) = b\eta^2 + \mu^2$, where $a, b\in\mathbb{Q}$ (we don't care about $\eta\mu$ because that is zero in $H^\ast(B\Diff_{x}(L)_0)$).
    By our identification of $H^\ast(B\Diff_{\{1\}}(L)_0)$ with $H^\ast(B\Diff_{\{1j\}}(L)_0)$ we have $\eta^2 + a\mu^2 = b\eta^2 + \mu^2$ and we conclude that $a = b = 1$.
\end{proof}

\begin{theorem}\label{thm: the rational cohommology of the main gorups identitiy component}
    Let $M\cong L_1\#L_2$ for two non-diffeomorphic generic lens spaces $L_1$ and $L_2$, fix a 3-disc in $L_1$ and $L_2$ to denote the discs that are cut out when connected summing, and $S^2$ in $M$ the sphere we join $L_1\setminus\interior{D^3}$ and $L_2\setminus\interior{D^3}$ along.
    Denote the rational cohomology groups     
    \[H^\ast(B\Diff(L_1\setminus\interior{D^3})_0) \cong \mathbb{Q}[\mu, \eta]/(\mu\eta)
    \;and\;
    H^\ast(B\Diff(L_2\setminus\interior{D^3})_0) \cong \mathbb{Q}[\nu, \vartheta]/(\nu\vartheta).\]    
    The map induced by the product of the restrictions
    \[H^\ast(B\Diff(L_2\setminus\interior{D^3})_0 \times B\Diff(L_1\setminus\interior{D^3})_0)\to H^\ast(B\Diff(M, S^2)_0)\]
    is surjective, and through it we obtain
    \[H^\ast(B\Diff(M, S^2)_0)\cong\mathbb{Q}[\mu, \eta,\nu, \vartheta]/(\mu\eta, \nu\vartheta, \mu^2+\eta^2 - \nu^2-\eta^2).\]
\end{theorem}
\begin{proof}
    Lemma \ref{lem: surj on cohom of fiber} implies that the comparison of the fibers in Figure \ref{fig:main ho fib seq comp} induces a surjection on rational cohomology.
    As the Leray-Serre spectral sequences induced by both of the rows in Figure \ref{fig:main ho fib seq comp} collapse on the $E_2$-page, the induced map on the total spaces
    \[f\colon H^\ast(B\Diff(L_2\setminus\interior{D^3})_0 \times B\Diff(L_1\setminus\interior{D^3})_0)\to H^\ast(B\Diff(M, S^2)_0).\]
    is surjective by naturality properties of the spectral sequences.
    
    This means that in order to figure out the cup product structure on $H^\ast(B\Diff_{D^3}(M)_0)$, we need to describe the kernel of this map.
    To compute this kernel we consider the square
   \[\begin{tikzcd}
    	{\Diff(M, S^2)_0} & {\Diff(L_1\setminus\interior{D^3})_0\times\Diff(L_2\setminus\interior{D^3})_0} \\
    	{\Diff(S^2)_0} & {\Diff(S^2)_0\times\Diff(S^2)_0.}
    	\arrow[from=1-1, to=1-2]
    	\arrow["{\text{res}^M_{S^2}}", from=1-1, to=2-1]
    	\arrow[from=1-2, to=2-2]
    	\arrow["\Delta", from=2-1, to=2-2]
    \end{tikzcd}\]
    This induces the maps on cohomology, where we will be interested in computing $g_1$ and $g_2$:
    \[\begin{tikzcd}
    	{H^\ast(B\Diff(M, S^2)_0)} & {H^\ast(B\Diff(L_1\setminus\interior{D^3})_0)\otimes_\mathbb{Q}H^\ast(B\Diff(L_2\setminus\interior{D^3})_0)} \\
    	{H^\ast(B\Diff(S^2)_0)} & {H^\ast(B\Diff(S^2)_0)\otimes_\mathbb{Q}H^\ast(B\Diff(S^2)_0).}
    	\arrow["f"', two heads, from=1-2, to=1-1]
    	\arrow[from=2-1, to=1-1]
    	\arrow["{g_1\otimes g_2}", from=2-2, to=1-2]
    	\arrow["\smile"', from=2-2, to=2-1]
    \end{tikzcd}\]
    Note that this diagram shows $f\circ (g_1\otimes g_2) = (\text{res}^M_{S^2})^\ast\circ\smile$.
    This means that since $f\circ (g_1\otimes g_2) = (\text{res}^M_{S^2})^\ast\circ\smile$,
    \[f(\text{pr}_1^\ast(g_1(p_1))\smile\text{pr}_2^\ast(g_2(1))) = (\text{res}^M_{S^2})^\ast(p_1)=
    f(\text{pr}_1^\ast(g_1(1))\smile\text{pr}_2^\ast(g_2(p_1))).\]
    And therefore $(g_1\otimes g_2)(p_1\otimes 1) - (g_1\otimes g_2)(1\otimes p_1)\in \ker(f)$.
    Since $g_1$ and $g_2$ are symmetric we will continue with the notation $g\colon H^\ast(B\Diff(S^2)_0)\to H^\ast(B\Diff(L\setminus \interior{D^3})_0)$.
    To understand this map we use the diffeomorphism group \[\Diff^{\text{SO}}(L, D^3) = \{\varphi\in \Diff(L, D^3)\,|\, \left.\varphi\right|_{D^3}\in \SO(3)\subseteq \Diff(D^3)\}\]
    consisting of those diffeomorphisms that rotate the 3-disc that is fixed set-wise.
    $\Diff^{\SO}(L, D^3)\simeq \Diff(L, D^3)$, as $\SO(3)\simeq \Diff(D^3)$.
    This fits into a diagram of the following form:
    \[\begin{tikzcd}
    	{\Diff(L\setminus \interior{D^3})_0} & {\Diff^{\text{SO}}(L, D^3)_0} & {\Diff_{\text{pt}}(L)_0} \\
    	{\Diff(S^2)_0} & {\SO(3)} & {\GL^+_3(\mathbb{R}).}
    	\arrow[from=1-1, to=2-1]
    	\arrow["\simeq"', from=1-2, to=1-1]
    	\arrow["\simeq", from=1-2, to=1-3]
    	\arrow[from=1-2, to=2-2]
    	\arrow["{T_{\text{pt}}}"', from=1-3, to=2-3]
    	\arrow["\simeq"', from=2-2, to=2-1]
    	\arrow["\simeq", from=2-2, to=2-3]
    \end{tikzcd}\]
    Here the top left horizontal map is an equivalence, because it is a composite of $\Diff^{\text{SO}}(L, D^3)_0\simeq \Diff(L, D^3)_0\simeq\Diff(L\setminus \interior{D^3})_0$.
    In Lemma \ref{lem: differential map on group cohomology} we have computed what the map induced on group cohomolgy by the map taking differentials at $\text{pt}$, the mid-point of the $D^3$ we cut out to get $L\setminus \interior{D^3}$.
    It sends $p_1$ to $\mu^2+\eta^2$ (for $L= L_1$).
    
    So getting back to computing the kernel of $f$, we can now see that
    $(\mu^2+\eta^2-\nu^2-\vartheta^2)\in \ker(f)$.
    We know the dimensions of $H^k(B\Diff(M, S^2)_0)$ for all $k$, and comparing dimensions we can see that this must be the whole kernel.

    Let us give a short argument for why the dimensions should agree.
    The dimensions in Theorem \ref{thm: main result} come from a spectral sequence that collapse on the $E_2$-page, with fiber $B\Diff_{L_2\setminus\interior{D^3}}(M)_0$ and base $B\Diff(L_2\setminus\interior{D^3})_0$.
    This means that the dimension of $H^k(B\Diff(M, L_2\setminus\interior{D^3}))_0$ is the same as the dimension of $H^k(B\Diff_{L_2\setminus\interior{D^3}}(M)_0\times B\Diff(L_2\setminus\interior{D^3})_0)$ as a $\mathbb{Q}$ vector space.
    So we wish to see that the dimension in each degree is the same for the graded $\mathbb{Q}$ vector spaces $\mathbb{Q}[\mu, \eta, \nu, \vartheta]/(\mu\eta, \nu\vartheta, \mu^2+\eta^2)$ and $\mathbb{Q}[\mu, \eta, \nu, \vartheta]/(\mu\eta, \nu\vartheta, \mu^2+\eta^2-\nu^2-\vartheta^2)$.
    Let us fix the lexicographic monomial order with $\mu> \eta> \nu> \vartheta$, and find the Gröbner basis with respect to this monomial order, now this shows that the leading term ideal of both of the ideals $I_1 = (\mu\eta, \nu\vartheta, \mu^2+\eta^2)$ and $I_2 = (\mu\eta, \nu\vartheta, \mu^2+\eta^2-\nu^2-\vartheta^2)$ is $\text{LT}(I_1) = \text{LT}(I_2) = (\mu\eta, \nu\vartheta, \eta^3, \mu^2)$.
    It is a fact of algebra, see e.g. \cite{CLO1}[Proposition 4, Chapter 5, Section 3] that as $\mathbb{Q}$ vector spaces $\mathbb{Q}[\mu, \eta, \nu, \vartheta]/I = \text{Span}(x^\alpha | x^\alpha\not \in \text{LT}(I))$ for any ideal $I\subseteq \mathbb{Q}[\mu, \eta, \nu, \vartheta]$.
\end{proof}

\section{The whole diffeomorphism group}
In our section on strategy we stated that $H^\ast(G)\cong H^\ast(G_0)^{\pi_0 G}$ but we are yet to describe the action in detail.
It is obtained as follows: take some element $[g]\in \pi_0 G$, and conjugation by this element induces a self map, $c_g$ of $H^\ast(BG_0)$, note that this construction is natural in the sense that given a group homomorphism $\varphi$, $H^\ast(B\varphi)\circ c_{\varphi(g)} = c_g\circ H^\ast(B\varphi)$ for all $g$ in the domain of $\varphi$.

In the following statement we use the notation form Theorem \ref{thm: the rational cohommology of the main gorups identitiy component}.
\begin{proposition}
    The action of $\pi_0\Diff(M)\cong C_2\times C_2$ on 
    \[H^\ast(B\Diff(M)_0)\cong \mathbb{Q}[\mu, \eta,\nu, \vartheta]/(\mu\eta, \nu\vartheta, \mu^2+\eta^2 - \nu^2-\eta^2)\]
    is generated by $c_{(-1, 1)}\colon \mu\mapsto -\mu$, $\eta\mapsto -\eta$ (leaving the other generators fixed), and $c_{(1, -1)}\colon\nu\mapsto -\nu$, $\vartheta \mapsto -\vartheta$.
\end{proposition}
\begin{proof}
From Theorem \ref{thm: mapping class group} we have the description $\pi_0 \Diff(M, S^2) \cong \pi_0 \Diff(L_1\setminus\interior{D^3}) \times \pi_0 \Diff(L_2\setminus\interior{D^3})$.
Since the product of the restrictions $\Diff(M, S^2)_0 \to \Diff(L_1\setminus\interior{D^3})_0\times \Diff(L_2\setminus\interior{D^3})_0$ induce a surjection on cohomology, using symmetry, we can figure out this action by investigating the action of $\pi_0\Diff(L\setminus \interior{D^3})$ on $H^\ast(\Diff(L\setminus \interior{D^3})_0)$.
But this through weak equivalences reduces to the action of $\pi_0\Diff_{\text{pt}}(L)$ on $H^\ast(\Diff_{\text{pt}}(L)_0)$.
Now from the calculation of the cohomology of $B\Diff_{\text{pt}}(L)_0$ we see that in fact through another surjection it is enough to figure out the action of $\pi_0\Isom(L)$ on $H^\ast(B\Isom(L)_0)$.

Without loss of generality we fix $\text{pt}\in L$ to be represented by $(1+0j)\in S^3$.
From the calculation of the isometries the element $f(j, j)$ represents the non-trivial element of $\pi_0\Isom(L)\cong C_2$.
Computation shows $F(j, j)\circ F(w_1, w_2) \circ F(j, j)^{-1} = F(\overline{w_1}, \overline{w_2})$.
This means that the action is generated by $\Isom(L)\cong S^1\times S^1 \to S^1\times S^1$, $(w_1, w_2)\mapsto (\overline{w_1}, \overline{w_2})$.

The conjugation map $S^1\to S^1$ on cohomology that sends the fundamental class $e$ to $-e$ by naturality of spectral sequences it also sends $e$ to $-e$ in the cohomology of $B S^1\cong \mathbb{CP}^\infty$.
\end{proof}
\begin{lemma}\label{lem: fixed points of a quotient}
    Let $G$ be group acting on a ring $R$ such that the order of $G$ is invertible in $R$.
    If $I\subseteq R$ is an ideal such that $I\subseteq G.I$, then 
    \[(R/I)^G\cong R^G/I^G\]
    where $I^G=R^G\cap I$ is an ideal of the subring $R^G\subseteq R$.
\end{lemma}
\begin{proof}
    Since $I\subseteq G.I$ and therefore the $G$ action on $R$ induces a $G$ action on $R/I$.
    Let us denote $I^G =I\cap R^G \subseteq R^G$, this is an ideal, we will see that $(R/I)^G \cong R^G/I^G$.
    The map $R^G/I^G\to R/I$, $[r]_{I^G}\mapsto [r]_{I}$ is well-defined, and injective since for $r, r^\prime \in R^G$, $[r]_{I^G} = [r^\prime]_{I^G}$ if and only if $(r-r^\prime)\subseteq I^G = R^G\cap I$, but $(r-r^\prime)\in R^G$ comes for free so this is the case if and only if $[r]_I = [r^\prime]_I$.
    The map $[r]_{I^G}\mapsto [r]_{I}$ lands in $(R/I)^G\subseteq R/I$, we will show that it hits all the elements of $(R/I)^G$.
    Consider some $[r]\in (R/I)$, we wish to see that if for all $g\in G$, $[g.r] = [r]$, then there is some $r^\prime \in R^G$, so that $[r] = [r^\prime]$.
    Set
    \[r^\prime = \frac{1}{|G|}\sum_{g\in G}g.r\]
    this is so that $[r^\prime] = \frac{1}{|G|}|G|\cdot [r] = [r]$.
    Clearly multiplication by any element of $G$ gives an isomorphism of $G$ thus we also have $r^\prime\in R^G$.
\end{proof}
From the theorem of Hatcher, Theorem \ref{theorem of Hatcher}, $H^\ast(B\Diff(M, S))\cong H^\ast(B\Diff(M))$ in the case where $M$ is the connected sum of two generic lens spaces.
In this light we state our main theorem with notations for $H^\ast(B\Diff(L_1\setminus\interior{D^3})_0)$ and $H^\ast(B\Diff(L_2\setminus\interior{D^3})_0)$ consistent with those in Theorem \ref{thm: the rational cohommology of the main gorups identitiy component}.
\begin{theorem}\label{main result}
    Let $M\cong L_1\#L_2$ for two non-diffeomorphic generic lens spaces $L_1$ and $L_2$, fix $D^3$ in $L_1$ and $L_2$ to denote the discs that are cut out when connected summing, and $S^2$ in $M$ the sphere we join $L_1\setminus\interior{D^3}$ and $L_2\setminus\interior{D^3}$ along.
    
    The map induced by the product of the restrictions
    \[H^\ast(B\Diff(L_2\setminus\interior{D^3})_0 \times B\Diff(L_1\setminus\interior{D^3})_0)\to H^\ast(B\Diff(M, S^2)_0)\]
    is surjective, and through it we obtain
    \[H^\ast(B\Diff(M, S^2)_0)\cong\mathbb{Q}[\mu, \eta,\nu, \vartheta]/(\mu\eta, \nu\vartheta, \mu^2+\eta^2 - \nu^2-\eta^2).\]
    Furthermore the composition $B\Diff(M, S)_0\to B\Diff(M, S)\to B\Diff(M)$ induces an inclusion of the cohomology of $B\Diff(M)$ as the subring
    \[H^\ast(B\Diff(M))\cong \mathbb{Q}[\mu^2, \eta^2, \nu^2, \vartheta^2] / (\mu^2\eta^2, \nu^2\vartheta^2, \mu^2+\eta^2-\nu^2-\vartheta^2).\]
\end{theorem}
\begin{proof}
    Let us denote $R = \mathbb{Q}[\mu, \eta, \nu, \vartheta]$, $G = C_2\times C_2$, and $I = (\mu\eta, \nu\vartheta, \mu^2+\eta^2 - \nu^2-\vartheta^2)$, these are such that they satisfy the assumptions of Lemma \ref{lem: fixed points of a quotient}.
    Calculation allows us to see that 
    \[R^G \cong \mathbb{Q}[\mu^2, \mu\eta, \eta^2, \nu^2, \nu\vartheta, \vartheta^2].\]
    This is because the fixed points of the action are polynomials whose terms only contain monomials $\mu^a\eta^b\nu^c\vartheta^d$ such that both $a+b$ and $c+d$ are even.
    All the generators of $I$ are in $R^G$ and therefore $I^G$ is generated by these same monomials.
\end{proof}
\bibliographystyle{amsalpha}
\bibliography{main}

\end{document}